\documentclass[12pt]{article}
\usepackage{amsmath,amsthm, amssymb, amsfonts, thmtools, enumitem, xspace}
\usepackage{mathrsfs}
\usepackage[sort&compress,numbers,comma,square]{natbib}
\usepackage[hidelinks]{hyperref}
\makeatletter
\g@addto@macro\bfseries{\boldmath}
\makeatother

\newtheorem{theorem}{Theorem}
\newtheorem{prop}[theorem]{Proposition}
\newtheorem{lemma}{Lemma}
\newtheorem{cor}[theorem]{Corollary}

\newlist{remarklist}{enumerate}{10}
\setlist[remarklist]{leftmargin=0ex, itemindent=1.5\parindent,
  label={\arabic*)}, itemsep=0ex}

\numberwithin{equation}{section}
\usepackage{environ}
\makeatletter
\NewEnviron{tagleft}{\tagsleft@true\BODY}  
\NewEnviron{tagright}{\tagsleft@false\BODY}
\NewEnviron{hide}{}                        
\makeatother

\long\def\comment#1{}  

\def\FT#1{\widehat#1}              
\def\wt#1{\widetilde#1}
\def\conj#1{\overline{#1}}  

\def\Im{\operatorname{Im}}               
\def\mean{\operatorname{E}}              
\def\pr{\operatorname{P}}                
\def\proj{\pi}                     
\def\Re{\operatorname{Re}}               
\def\rx{\epsilon}                  %
\def\sgn#1{\operatorname{sign}#1}
\def\sph#1{\mathbb{S}^{#1}}        
\def\sint{\int_{\sph{d-1}}}
\def\rint{\int_{\Reals^d}}
\def\Iff{\Longleftrightarrow}   

\def\cf#1{\mathbf{1}\!\Cbr{#1}}

\def\Grp#1{\left(#1\right)}
\def\Cbr#1{\left\{#1\right\}}
\def\Sbr#1{\left[#1\right]}
\def\Cil#1{\left\lceil#1\right\rceil}
\def\Abs#1{\left|#1\right|}
\def\Flr#1{\left\lfloor#1\right\rfloor}
\def\Norm#1{\left\|#1\right\|}
\def\ip#1#2{\langle #1,#2\rangle}  

\def\dto{\downarrow}
\def\toi{\to\infty}
\def\dd{\mathrm{d}}
\def\iunit{\mathrm{i}}

\def\dg{q}
\def\nth#1{\frac{1}{#1}}
\def\lfrac#1#2{#1/#2}     
\def\th{\sp{\rm th}}

\def\Sp#1{\sp{(#1)}}

\def\Coms{\mathbb{C}}
\def\Nats{\mathbb{N}}
\def\Ints{\mathbb{Z}}
\def\Reals{\mathbb{R}}
\def\intzi{\int_0^\infty}
\def\sumoi#1{\sum_{#1=1}^\infty}
\def\sumzi#1{\sum_{#1=0}^\infty}
\def\eno#1#2{#1_1, \ldots, #1_{#2}}             

\def\Cal#1{\mathcal{#1}}
\makeatletter
\def\@cleandot{\@ifnextchar.{}{\@ifnextchar,{.}{\@ifnextchar;{.}{\@ifnextchar?{.}{\@ifnextchar:{.}{\@ifnextchar!{.}{\@ifnextchar'{.}{\@ifnextchar){.}{.\ }}}}}}}}}
\def\pdf{p.d.f\@cleandot}
\def\lhs{l.h.s\@cleandot}
\def\rhs{r.h.s\@cleandot}
\def\ac{a.c\@cleandot}
\def\wrt{w.r.t\@cleandot}

\def\const{\mathsf{const}}
\makeatother
\def\levy{L\'evy\xspace}
\begin{document}
\begin{center}
  \large
  \textbf{Spherical harmonic analysis for multivariate stable
    distributions\footnote{Research partially supported by NSF Grant
      DMS 1720218.}
  }
  \\[1.5ex]
  \normalsize
  Zhiyi Chi \\
  Department of Statistics\\
  University of Connecticut \\
  Storrs, CT 06269, USA, \\[.5ex]
  E-mail: zhiyi.chi@uconn.edu \\[1ex]
  \today
\end{center}

\begin{abstract}
  Series representations consisting of spherical harmonics are obtained
  for characteristic exponents and probability density functions of
  multivariate stable distributions under various conditions.  A
  result potentially applicable in a practical setting is that for any 
  distribution with stability index not equal to 1 and with a
  polynomial spectral spherical density, the series representation
  converges absolutely with all terms being calculable in closed
  form.  Asymptotic expansions consisting of spherical harmonics are
  also considered for probability density functions.

  \medbreak
  
  \medbreak

  \medbreak
  \textbf{Acknowledgment.}  The research is partially supported by
  NSF Grant DMS 1720218. 
\end{abstract}

\section{Introduction} \label{s:intro}
The probability density functions (\pdf's) of univariate stable
distributions, i.e., stable distributions on $\Reals$, have been well
known for a long time.  In contrast, much less can be said about the
\pdf's of multivariate stable distributions except for a few cases,
such as spherically symmetric stable distributions, or more generally,
subordinated normal distributions  \cite {uchaikin:99:vsp}, and direct
products of such distributions and univariate stable distributions.
Many efforts have been dedicated to the understanding of multivariate
stable distributions; see \cite{nolan:98:pght} for an
review.  One approach is to approximate the distributions by more
tractable ones, such as series of simple random variables or stable
distributions with discrete spectral spherical measures
\cite{davydov:02:ltps, davydov:02:jma, byczkowski:93:jma}.  This
approach provides error bounds of approximation but not functional
forms of the distributions.  On the other hand, \cite
{abdul-hamid:98:jma} gives integral expressions for multivariate
stable \pdf's and \cite {fallahgoul:14:jstp} provides analytic
approximations of the \pdf's by solutions of partial differential
equations of fractional order. Despite the progress, it has been
difficult to extend several important representations for univariate
stable distributions (\cite {uchaikin:99:vsp}, chapter 4) to
the multivariate ones.  To a large degree, the difficulty is due to a
lack of analytic tools.  However, in the study on estimation for
multivariate stable distributions, spherical harmonic analysis has
already been used \cite {pivato:03:jma}.  Inspired by this, the aim of
the paper is to apply spherical harmonic analysis to get more
understanding of the \pdf's of multivariate stable distributions.

Let $\mu$ be an $\alpha$-stable probability distribution on
$\Reals^d$ with no shift, where $\alpha\in (0,2)$ is the stability
index.  Unless $\mu$ is a unit mass at 0, there is a finite nonzero
Borel measure $\lambda$ on the unit sphere $\sph{d-1} =
\{x\in\Reals^d{:}~|x| = 1\}$, where $|x|$ denotes the Euclidean norm,
such that the Fourier transform of $\mu$ is (\cite{sato:99:cup},
Theorem 14.10)
\begin{align*}
  \FT\mu(z)
  &= \rint e^{\iunit\ip x z}\mu(\dd x) = \exp\{-\Phi_\mu(z)
    \}, \quad
    z\in\Reals^d,
\end{align*}
where
\begin{align}\label{e:che-ss}
  \Phi_\mu(z)
  =
  \begin{cases}
    \displaystyle\int_{\sph{d-1}} |\ip z v|^\alpha
    \Sbr{1-\iunit\tan\frac{\pi\alpha} 2 \sgn{\ip z v}}\!
    \lambda(\dd v)& \text{if } \alpha\ne1
    \\[2ex]
    \displaystyle\int_{\sph{d-1}}
    \Sbr{|\ip z v| + \iunit \frac{2}{\pi} \ip z v\ln|\ip z v|}\!
    \lambda(\dd v)
    & \text{if } \alpha=1
  \end{cases}
\end{align}
is known as the characteristic exponent of $\mu$.  The measure
$\lambda$ is unique (\cite {bochner:55:ucp}, Theorem 3.4.2) and will
be referred to as the spectral spherical measure of $\mu$, although it
has been called the spectral or Poisson spectral measure elsewhere
(cf.~\cite{davydov:02:jma}).  If the support of $\mu$ is not contained
in $a+E$ for any $a\in\Reals^d$ and linear subspace $E\subset
\Reals^d$ with $\dim(E)<d$, then $\mu$ is called nondegenerate and has
a bounded continuous \pdf with respect to (\wrt)\ the Lebesgue measure
(\cite{sato:99:cup},  Definition 24.16 and Example 28.2)
\begin{align}\label{e:ift-ss}
  g(x)=
  (2\pi)^{-d} \rint e^{-\iunit\ip xz} \FT\mu(z)\,\dd z.
\end{align}
From \eqref{e:che-ss} and \eqref{e:ift-ss}, the calculation of the
\pdf boils down to two integrals, one for $\Phi_\mu$, the other for
the inverse Fourier transform of $\exp\{-\Phi_\mu\}$.  To tackle the
integrals, a significant portion of the paper will consider the case
where $\lambda$ has a density \wrt the Haar measure on $\sph{d-1}$,
henceforth referred to as the spectral spherical density.

Section \ref{s:prelim} sets up notation.  It also lists basic facts
about spherical harmonics and several other special functions,
all selected from the classical books \cite {andrews:99:cup} and
\cite{stein:71:pup}.  Section \ref{s:che} considers the characteristic
exponential $\Phi_\mu$ of a nondegenerate $\alpha$-stable distribution
$\mu$ that has no shift.  In general, if $\alpha\ne1$, then
$\Phi_\mu(z)$ can be written as $|z|^\alpha V(u_z)$, where $u_z$ is
the unit vector with the same direction as $z$.  Under the condition
that $\mu$ has a square-integrable spectral spherical density $P$, the
section shows that $V$ can be expressed as a series of spherical
harmonics, each being an explicit multiple of a spherical harmonic of
$P$.  As a result, if $P$ is a polynomial, then $\Phi_\mu(z)$ has a
closed form.  Similar results also hold when $\alpha=1$.  Thus, it is
possible to get the characteristic exponential in a practical setting,
although a direct calculation typically is tedious; see comments in
Section \ref{ss:closedform}.

Sections \ref{s:pdf(0,1)} and \ref{s:pdf(1,2)} consider the \pdf
$g(x)$ for nondegenerate $\alpha$-stable distributions with
$\alpha\ne1$ and no shift.  In both sections, some emphasis is given
to series representations that are absolutely convergent (\ac) and
have all terms calculable in closed form.  Section \ref{s:pdf(0,1)}
deals with the case $\alpha\in(0,1)$.  It shows that $g(x)$ can be
represented as an infinite series of spherical harmonics, each coming
from a positive integral power of the aforementioned $V$ and weighted
by a negative fractional power of $|x|$.  When $\mu$ has a polynomial
spectral spherical density, the series is uniformly \ac in
$\{x{:}~|x|\ge r\}$ for any $r>0$ and all the spherical harmonics can
be written in closed form.  On the other hand, for the general case
where $\mu$ may not have a spectral spherical density, only
convergence in an $L^2$ sense is established, which nevertheless is
strong enough to yield the \pdf of  $|X|$ with $X\sim\mu$.  It is of
interest to consider cases in between the above two.  For example,
when $\mu$ has a square-integrable spectral spherical density that is
not a polynomial, it would be desirable to have an \ac series of
spherical harmonics for $g(x)$.  Results of this sort will require a
better understanding of the spherical harmonics of high powers of $V$.
The asymptotic expansion of $g(x)$ at $0$ is also derived in this
section.  In contrast to the series representations, the asymptotic
expansion consists of spherical harmonics that come from negative
powers of $V$ and are weighted by nonnegative integral powers of
$|x|$.

Section \ref{s:pdf(1,2)} deals with the case $\alpha\in (1,2)$.
and furnishes two \ac series representations of $g(x)$ in terms of
spherical harmonics.  The first one consists of spherical harmonics
coming from negative fractional powers of $V$.  However, these
spherical harmonics do not have easily available closed form even when
$V$ is a nonconstant polynomial.  The second series representation
allows all its spherical harmonic terms to be expressed  in closed
form when $\mu$ has a polynomial spectral spherical density.  However,
it requires a somewhat arbitrary parameter.  The section also
considers the asymptotic expansion of $g(x)$ as $|x|\toi$.  Compared
to the same problem in the univariate case, the analysis is much more
subtle.  The section obtains the asymptotic expansion of the spherical
harmonic of $g$ of every fixed degree.  There is still a significant
gap between the result and a full asymptotic expansion of $g$,
although the former strongly suggests the latter.  On the other hand,
the result provides useful information such as an asymptotic expansion
of the \pdf of $|X|$ with $X\sim\mu$.

As applications of the above results, Section \ref{s:example}
illustrates the case where the spectral spherical density is a linear
function and Section \ref{s:sampling} shows that the series
representations can be applied to sample multivariate stable
distributions.  However, the important issue of efficient sampling,
which is available for univariate stable \pdf's
(cf.~\cite{devroye:86:sv-ny}, section IV.6), is beyond the scope of
the paper.  Also, the paper has no result on the \pdf for $\alpha=1$.
In view of currently available results in the univariate case (\cite
{sato:99:cup}, p.~88), it is possible that the multivariate case for  $\alpha=1$ does not admit a simple series  representation.

\section{Preliminaries} \label{s:prelim}
Denote $\Nats=\{1,2,\ldots\}$, $\Ints_+ = \{0\}\cup \Nats$, and
$\omega = \omega_{d-1}$ the measure on $\sph{d-1}$ such that for $f\in
L^1(\Reals^d)$, $\rint f(x) \,\dd x = \sint [\intzi f(s v)
s^{d-1}\,\dd s]\,\omega(\dd v)$.  For $0\ne x\in\Reals^d$, let
$u_x=x/|x|$.  Since $u_x$ will be used only when $x\ne0$ or in
functions of the form $|x|^a f(u_x)$ with $a>0$ and $f$ bounded, $u_0$
need not be specified.

\subsection{Spherical harmonics} \label{ss:sh}
A polynomial on $\Reals^d$ of degree $j\in\Ints_+$ has the form $f(x)
= \sum c_a x^a$, $x = (\eno x d)\in \Reals^d$, where the sum is taken
over all $a = (\eno a d)\in \Ints^d_+$ with $\sum a_i\le j$, $x^a$
denotes $\prod x^{a_i}_i$, and each $c_a\in\Coms$ is a constant.  If
$c_a=0$ unless $\sum a_i=j$, then $f$ is said to be homogeneous of
degree $j$ and $f(x) = |x|^j f(u_x)$ for $x\ne0$.  If $\sum
\frac{\partial^2 f}{\partial x^2_i}\equiv0$, then $f$ is said to be
harmonic.  Restrictions of homogeneous harmonic polynomials to
$\sph{d-1}$ are called spherical harmonics (\cite{andrews:99:cup},
Section 9.4).  Due to $|x|^2=1$ on $\sph{d-1}$, different polynomials
when restricted to $\sph{d-1}$ can be equal, e.g., $f(x)$ and $|x|^2
f(x)$.  However, a spherical harmonic is the restriction of a unique
homogeneous harmonic polynomial on $\Reals^d$ and the two have the
same degree; see \cite{andrews:99:cup}, Theorem 9.4.1 or
\cite{stein:71:pup}, Corollary VI.2.4.

Recall that $L^2(\sph{d-1})$ is equipped with inner product $\ip f g =
\sint f \conj g\,\dd\omega$, where $\conj{\phantom{x}}$ stands for
complex conjugation.  Denote by $\Cal P_{j,d}$ the set of
polynomials of degree $j$ restricted to $\sph{d-1}$, and $\Cal
H_{j,d}$ that of spherical harmonics of degree $j$.  Both $\Cal
P_{j,d}$ and $\Cal H_{j,d}$ are finite dimensional subspaces of
$L^2(\sph{d-1})$.  From \cite{stein:71:pup}, p.~140 or \cite
{andrews:99:cup}, p.~450,
\begin{align} \label{e:dim-h}
  \dim(\Cal H_{j,d}) = c_{j,d} =
  \begin{cases}
    1, & j=0 \\
    \binom{j+d-2}{j} +
    \binom{j+d-3}{j-1}, & j>0,
  \end{cases}
\end{align}
$\Cal H_{j,d}\perp \Cal H_{j',d}$ for $j\ne j'$, $\Cal P_{k,d} =
\oplus^k_{j=0} \Cal H_{j,d}$ for $k\ge 0$, and $L^2(\sph{d-1}) = \oplus^\infty_{j=0} \Cal H_{j,d}$.

The projection from $L^2(\sph{d-1})$ to $\Cal H_{j,d}$ will be denoted 
by $\proj_j$.  For $f\in L^2(\sph{d-1})$, $\proj_j f$ will be referred
to as the $j$th spherical harmonic of $f$.  If $d\ge2$, then given
$j\ge 0$ and any real-valued
orthonormal basis $S_{j,i}$, $i=1,\cdots, c_{j,d}$, of $\Cal H_{j,d}$,
for $f\in L^2(\sph{d-1})$, $(\proj_j f)(u) = \sint
\sum^{c_{j,d}}_{i=1} S_{j,i}(u) S_{j,i}(v) f(v)\,\omega(\dd  v)$.  If
$d>2$, then by \cite{andrews:99:cup}, Theorem 9.6.3,
\begin{align} \label{e:proj}
  (\proj_j f)(u) = \frac{c_{j,d}}{A(\sph{d-1})}
  \sint
  \wt C^{(d-2)/2}_j(\ip u v) f(v)\, \omega(\dd v), \quad
  u\in\sph{d-1},
\end{align}
where $A(\sph{d-1}) = \lfrac{2\pi^{d/2}}{\Gamma(d/2)}$ is the surface
area of $\sph{d-1}$ and for $b>0$,
\begin{align*}
  \wt C^b_j(t) = \lfrac{C^b_j(t)}{C^b_j(1)}
\end{align*}
is an ultraspherical polynomial with $C^b_j(t)$ known as the Gegenbauer
polynomial of degree $j$ (\cite{andrews:99:cup}, p.~302).  The latter
can be written as
\begin{align} \label{e:Gegenbauer0}
  C^b_j(t) =\sum^{\Flr{j/2}}_{m=0} (-1)^m 2^{j-2m}
  \frac{(b)_{j-m}}{m! (j-2m)!} t^{j-2m}
\end{align}
(\cite{pivato:03:jma}, p.~233) and for $|t|\le1$, $|C^b_j(t)|\le
C^b_j(1) = \lfrac{(2b)_j}{j!}$, where for $z\in\Coms$ and
$n\in\Nats$, $(z)_0=1$ and $(z)_n= \prod^{n-1}_{m=0} (z+m)$.
For $d=2$, if $j\ge1$, then $C^0_j(t)\equiv0$, so $\wt C^0_j(t)$
cannot be defined as the ratio of $C^0_j(t)$ and $C^0_j(1)$.  However,
$\lim_{b\dto0} C^b_j(t)/C^b_1(t) = T_j(t)$ (\cite {andrews:99:cup},
Eq.~(6.4.13$'$)), where $T_j(t)$ is the Tchebyshev polynomial of the
first kind of degree $j$ defined by the formula
$T_j(\cos\theta) = \cos(j\theta)$ (\cite {andrews:99:cup}, p.~101).
Then \eqref {e:proj} still holds with $\wt C^0_j(t) = T_j(t)$
(cf.\ \cite{andrews:99:cup}, Remark 9.6.1).  It is useful to note that
\begin{align} \label{e:Gegenbauer}
  |\wt C^b_j(t)|\le1, \quad -1\le t\le 1,\ b\ge0,\ j\in\Ints_+.
\end{align}  

From \eqref{e:proj}, $\proj_j$ is an integral operator with kernel
$K(u,v) = \frac{c_{j,d}}{A(\sph{d-1})} \wt C^{(d-2)/2}_j(\ip u v)$.
Given $u\in\sph{d-1}$, the polynomial $\wt C^b_j(\ip u\cdot)$ is known
as a zonal harmonic with pole $u$ and belongs to $\Cal H_{j,d}$
(\cite  {andrews:99:cup}, p.~455--456).  The Funk-Hecke formula (\cite
{andrews:99:cup}, Theorem 9.7.1) states that, for any continuous
function $f$ on $[-1,1]$ and any $S\in \Cal H_{j,d}$, $j\in\Ints_+$,
$d\ge2$,
\begin{gather} \label{e:Funk-Hecke0}
  \begin{split}
    &
    \sint f(\ip u v) S(v)\,\omega(\dd v) = \lambda_{j,d}
    S(u), \quad u\in \sph{d-1}
    \\\text{with}~\lambda_{j,d}
    &= \lambda_{j,d}(f)
    =
    A(\sph{d-2}) \int^1_{-1} f(t) \wt C^{(d-2)/2}_j(t)
    (1-t^2)^{(d-3)/2}\, \dd t.
  \end{split}
\end{gather}

\subsection{Closed form of spherical harmonics of a
  polynomial} \label{ss:closedform}
Let $f$ be a polynomial.  Then $\wt C^{(d-2)/2}_j(\ip u v) f(v)$
can be written as a linear combination of a finite number of $\ip u
v^k v^c$ with explicit coefficients, where $k\in\Ints_+$ and $c\in
\Ints^d_+$.  By expanding $\ip u v^k v^c$ and integrating term by
term, \eqref {e:proj} can be written as
\begin{align*}
  (\proj_j f)(u)=
  \sum m_{a,b} \sint v^a \omega(\dd v)\cdot u^b,
\end{align*}
where the sum only has a finite number of terms, and for each pair
$a$, $b\in\Ints^d_+$, $m_{a,b}$ is an explicit number.  In polar
coordinates, $v_i =  \cos\theta_i\prod^{i-1}_{j=1} \sin\theta_j$ for
$i<d$ and $v_d = \prod^{d-1}_{j=1} \sin\theta_j$, where $\theta =
(\eno \theta {d-1})\in E = [0,\pi]^{d-2}\times [0,2\pi]$.  Then
(\cite {andrews:99:cup}, Eq.~(9.6.4))  \label{e:trigom}
\begin{gather*}
  \omega(\dd v) = \prod^{d-2}_{j =1} (\sin\theta_j)^{d-1-j}
  \,\dd\theta
  = (\sin\theta_1)^{d-2} \cdots (\sin\theta_{d-3})^2 \sin \theta_{d-2}
  \,\dd\theta_1\cdots\dd\theta_{d-2}\,\dd\theta_{d-1}
  \\
  \text{and}\qquad
  \sint v^a \omega(\dd v)
  =
  \int_E\prod^{d-1}_{j=1} (\sin\theta_j)^{p_j} (\cos\theta_j)^{q_j}\,
  \dd\theta
\end{gather*}
with $p_j = p_j(a)$ and $q_j = q_j(a)\in \Ints_+$.  Thus $\proj_j f$
can be found in closed form by trigonometric integration.
Unfortunately, if $f$ has a high degree or $j$ is large, the
calculation involves a large number of $\sint v^a \omega(\dd v)$ and
becomes tedious.

For example, let $d=3$ and $f(u) = f(u_1, u_2, u_3) = u^2_1$.
Then $\proj_j f=0$ for $j>2$.  To find $\proj_2 f$, from
\eqref{e:dim-h}, $c_{2,3} =5$ and from \eqref{e:Gegenbauer0}, $\wt
C^{1/2}_2(t) =(3t^2-1)/2$.  By \eqref{e:proj},
\begin{align*}
  (\proj_2 f)(u)
  =
  \frac5{4\pi} \int_{\sph2}
  \frac12[3(u_1 v_1 + u_2 v_2 + u_3 v_3)^2-1] v^2_1\,\omega(\dd v).
\end{align*}
As described above, let $v_1 = \cos\theta$, $v_2 = \sin\theta
\cos\phi$, and $v_3 =\sin\theta \sin\phi$, $\theta\in
[0,\pi]$, $\phi\in [0,2\pi]$.  Routine trigonometric integration gives $\int_{\sph2} v^4_1\, \omega(\dd
v) = \int^{2\pi}_0 \int^\pi_0 \cos^4\theta \sin\theta\, \dd\theta\,
\dd\phi = \lfrac{4\pi}5$,  $\int_{\sph2}
v^2_1\,\omega(\dd v) = \lfrac{4\pi}3$, and for $i=1,2$, $\int_{\sph2}
v^2_1 v^2_i \,\omega(\dd v) = \lfrac{4\pi}{15}$ and $\int_{\sph2}
v^2_1 v_2 v_3\,\omega(\dd v)= \int_{\sph2} v^3_1 v_i\,\omega(\dd v)
=0$.  Then $(\proj_2 f)(u) =(3u^2_1 + u^2_2 + u^2_3)/2 - \lfrac56$.
This may seem contrary to $\proj_2 f$ being a harmonic homogeneous
polynomial restricted to $\sph2$, but is in fact correct as $(\proj_2
f)(u)=(2 u^2_1 - u^2_2 - u^2_3)/3$ for $u\in\sph2$.  More calculation
yields $\proj_1 f=0$ and $\proj_0 f = 1/3$.

For this example, simpler calculation can be made using the Funk-Hecke
formula \eqref{e:Funk-Hecke0}.   Let $e=(1,0,0)$ and $g(t) = t^2$.
Then $(\proj_j f)(u) = (\lfrac5{4\pi}) \int_{\sph2} g(\ip e v)\wt
C^{1/2}_2(\ip u v) \,\omega(\dd v)$.   Given $u$, $\wt C^{1/2}_2(\ip
u\cdot)\in\Cal H_{2,3}$, so $(\proj_j f)(u) = (5/4\pi)
\lambda_{2,3}\wt C^{1/2}_2(\ip u e) = \lfrac{(3 u^2_1-1)}3$ with  $\lambda_{2,3} = A(\sph1) \int^{-1}_1 t^2[(3t^2-1)/2]\,\dd t= \lfrac{8\pi}{15}$.

\subsection{Other special functions}
For $\eno a p$, $\eno b q\in\Coms$, $p$, $q\ge0$, with
$b_i\not\in\{0,-1,-2, \ldots\}$, the series
\begin{align*}
  \sumzi k \frac{(a_1)_k \cdots (a_p)_k z^k}{(b_1)_k \cdots (b_q)_k
  k!},
\end{align*}
where $(a_1)_k \cdots (a_p)_k =1$ if $p=0$ and likewise for
$(b_1)_k\cdots (b_q)_k$, is known as the hypergeometric series and
denoted by ${_p}F_q(\eno a p; \eno b q; z)$ or
\begin{align*}
  {_p}F_q
  \Grp{
    \begin{matrix}
      \eno a p \\ \eno b q
    \end{matrix}; z
  }.
\end{align*}
The series is \ac for all $z\in\Coms$ if $p\le q$, and for $|z|<1$ if
$p=q+1$.   When $p=2$ and $q=1$, the analytic
continuation of the series $_2F_1(a,b; c; z)$ as a function of $z$ is
known as the hypergeometric function.  Gauss' formula states that
(\cite{andrews:99:cup}, Theorem 2.2.2)
\begin{align}\label{e:Gauss2F1}
  {_2}F_1
  \Grp{
    \begin{matrix}
      a, b \\ c
    \end{matrix}; 1
  } = \frac{\Gamma(c) \Gamma(c-a-b)}{\Gamma(c-a) \Gamma(c-b)} \quad
  \text{if~} \Re(c-a-b)>0.
\end{align}

Several classical formulas for the Gamma function will be used 
(\cite{andrews:99:cup}, Chapter 1):
\begin{alignat}{4} \label{e:L-dup}
  &\text{(Legendre's duplication)}&&\quad
  \Gamma\Grp{\frac z2} \Gamma\Grp{\frac{z+1}2} =
  2^{1-z} \Gamma(z)\sqrt \pi,
  \\\label{e:Euler}
  &\text{(Euler's reflection)} &&\quad
  \Gamma(z) \Gamma(1-z) = \frac\pi{\sin(\pi z)},
  \\\label{e:Stirling}
  &\text{(Stirling)}&&\quad
  \Gamma(z) \sim
  \sqrt{2\pi/z} (z/e)^z, \quad \Re z\toi.
\end{alignat}
From Stirling's formula, given $c\in\Reals$, $\Gamma(z+c)/\Gamma(z)
\sim z^c$.  Then from \eqref{e:dim-h}
\begin{align} \label{e:c-asym}
  c_{j,d} = O(j^{d-2}), \quad j\toi.
\end{align}
Since $\Gamma(z) = \intzi t^{z-1} e^{-t}\,\dd t$ if $\Re z>0$, by
H\"older's inequality, the Gamma function is log-convex on
$(0,\infty)$, i.e., given $c>0$, $\Gamma(z+c)/\Gamma(z)$ is
increasing in $z\in (0,\infty)$.  The reciprocal of the Gamma function
can be continuously extended into an entire function whose set of
zeros is exactly the set of negative integers; this continuation is
still denoted by $1/\Gamma(x)$ (\cite{andrews:99:cup}, p.~3).  Thus,
for functions of the form $f=g/\Gamma$, if $g$ has a finite value at
$-x$ with $x\in\Nats$, then $f$ is well defined at $-x$ with value
zero.

The Bessel function of the first kind of order $a$ is defined by
(\cite{andrews:99:cup}, Eq.~(4.5.3))
\begin{align} \label{e:Bessel}
  J_a(x) = \frac{(x/2)^a}{\Gamma(a+1)}
  \,{_0}F_1
  \Grp{
    \begin{matrix}
      \text{---} \\ a+1
    \end{matrix}; -(x/2)^2
  } =
  \sumzi k \frac{(-1)^k (x/2)^{2k+a}}{k! \Gamma(k+a+1)}.
\end{align}
It is known that (\cite{NIST:10}, 10.14.4)
\begin{gather}
  \label{e:Bessel3}
  |J_a(z)| \le \frac{|z/2|^a e^{|\Im z|}}{\Gamma(a+1)}, \quad a\ge
  -1/2, \ z\in\Coms.
\end{gather}

Finally, let $d\ge2$.  For $r>0$, $u,v\in\sph{d-1}$, and $S\in \Cal
H_{j,d}$, $j\ge0$, letting $b=d/2-1$,
\begin{align*}
  \sint e^{-\iunit r\ip u v} S(v)\,\omega(\dd v)
  = (2\pi)^{d/2} (-\iunit)^j S(u) \frac{J_{j+b}(r)}{r^b}.
\end{align*}
This is essentially Eq.~(9.10.2) in \cite{andrews:99:cup},
except that the latter incorrectly uses factor $\iunit^j$ instead of
$(-\iunit)^j$ on the  \rhs.  If $f\in L^2(\sph{d-1})$, then from $f =
\sumzi j \proj_j f$ and the above formula,
\begin{align} \label{e:ft-j}
  \sint e^{-\iunit r\ip u v} f(v)\,\omega(\dd v)
  =
  \frac{(2\pi)^{d/2}}{r^b}
  \sumzi j (-\iunit)^j J_{j+b}(r) (\proj_j f)(u).
\end{align}
In particular, for any $w\in\sph{d-1}$, since $\wt C^b_j(\ip w
\cdot)\in \Cal H_{j,d}$, then \eqref{e:proj} and \eqref{e:ft-j} 
yield
\begin{align}\nonumber
  (\proj_j e^{-\iunit r \ip u \cdot})(w)
  &=
    \frac{c_{j,d}}{A(\sph{d-1})}
    \sint
    e^{-\iunit r\ip u v} \wt C^{(d-2)/2}_j(\ip w v) \,\omega(\dd v)
  \\\label{e:FJ-proj}
  &=
    \frac{(2\pi)^{d/2}}{A(\sph{d-1})}  (-\iunit)^j c_{j,d}
    \wt C^b_j(\ip u w) \frac{J_{j+b}(r)}{r^b}.
\end{align}

\section{Calculation of characteristic exponent} \label{s:che}
Throughout the section, let $d\ge2$.  The main result of this section
is the following.
\begin{theorem} \label{t:che}
  Let the characteristic exponent $\Phi_\mu$ of an
  $\alpha$-stable distribution $\mu$ be given by \eqref{e:che-ss}.
  Suppose $\lambda(\dd v) = P(v)\,\omega(\dd v)$ with $P\in
  L^2(\sph{d-1})$.  Let $P_j = \proj_j P$, $j\in\Ints_+$.  Then 
  \begin{align}\label{e:che-ne1}
    \Phi_\mu(z)=
    |z|^\alpha \Sbr{\sum_{j\mathrm{~even}}
    w_j(\alpha) P_j(u_z)
    -\iunit\tan\frac{\pi\alpha} 2
    \sum_{j\mathrm{~odd}}
    w_j(\alpha) P_j(u_z)},
  \end{align}
  if $\alpha\ne1$, and 
  \begin{align} \label{e:che-eq1}
    \Phi_\mu(z)
    =
    |z|
    \Sbr{\sum_{j\mathrm{~even}} w_j(1) P_j(u_z)
    + \iunit\frac2\pi \sum_{j\mathrm{~odd}}
    w^*_j P_j(u_z)}
    + \iunit |z|\ln |z|
    \frac{2\pi^{d/2-1}}{\Gamma(d/2+1)}
    P_1(u_z)
  \end{align}
  if $\alpha=1$, where
  \begin{gather} \label{e:Funk-Hecke2}
    \begin{split}
      w_j(\alpha)
      &=
      \frac{\pi^{d/2} \Gamma(\alpha+1)}{2^{\alpha-1}
        \Gamma((j+\alpha+d)/2) \Gamma((\alpha-j)/2+1)}
      \\
      &=
      \frac{\pi^{d/2-1} \sin((j-\alpha)\pi/2)
        \Gamma(\alpha+1)\Gamma((j-\alpha)/2)}{2^{\alpha-1}
        \Gamma((j+\alpha+d)/2)},
    \end{split}
  \end{gather}
  and
  \begin{gather}  \label{e:che-eq1d}
    w^*_j=
    \begin{cases}
      \displaystyle
      \frac{\pi^{d/2}}{\Gamma(d/2+1)}\beta_d & j=1
      \\[1ex]
      \displaystyle
      \frac{(-1)^{(j-3)/2}\pi^{d/2}\Gamma((j-1)/2)}
      {2\Gamma((d+1+j)/2)} & j>1~\text{odd},
    \end{cases}
  \end{gather}
  where, letting $s_0=0$ and $s_n =
  s_{n-1} +1/n$ for $n\ge1$,
  \begin{align*}
    \beta_d =
    \begin{cases}
      1-\ln2 - s_{d/2}/2 & \text{~$d$ even,} \\
      1-s_{d+1} + s_{(d+1)/2}/2 & \text{~else}.
    \end{cases}
  \end{align*}
  The series in \eqref{e:che-ne1} and \eqref{e:che-eq1} are \ac. 
\end{theorem}

\begin{proof}[Remark]{\ }
  \begin{remarklist}
  \item
    If $\alpha=1$ and $\mu$ is strictly stable, then by
    \cite{sato:99:cup}, Theorem 14.10, $\sint v \lambda(\dd
    v)= \sint v P(v)\,\omega(\dd v)=0$.  Consequently, by
    \eqref{e:proj}, $P_1\equiv0$ in \eqref{e:che-eq1}.

  \item The two expressions in \eqref{e:Funk-Hecke2} are equal by
    Euler's reflection formula \eqref{e:Euler}.  Both are presented
    because either one may be more convenient to use in certain cases.

  \item
    If $P$ is a polynomial of degree $\dg$, then from Section
    \ref {s:prelim}, $P_j = \proj_j P = 0$ for $j>\dg$, and $P_j$
    with $j\le\dg$ can be explicitly calculated in closed form by
    \eqref{e:proj}.  As a result, $\Phi_\mu(z)$ can be obtained in
    closed form.  See Section \ref{s:example} for an example.
    \qedhere
  \end{remarklist}
\end{proof}

From \eqref{e:che-ss}, if $\mu$ has no shift and $\alpha\ne1$,
then
\begin{align} \label{e:che-ss3}
  \begin{array}{c}
    \Phi_\mu(z) = |z|^\alpha V(u_z),\\[1.5ex]
    \text{with~}\displaystyle
    V(u) =
    \sint |\ip u v|^\alpha\Sbr{1-\iunit\tan\frac{\pi\alpha} 2
      \sgn{\ip u v}}\! \lambda(\dd v).
  \end{array}
\end{align}
The following is a standard result.
\begin{lemma} \label{l:ReV}
  If $\mu$ has characteristic exponent \eqref{e:che-ss3}, then
  \begin{align*}
    \sup_{\sph{d-1}} |V| \le \frac{\lambda(\sph{d-1})}
    {|\cos(\pi\alpha/2)|} < \infty.
  \end{align*}
  and the distribution is  nondegenerate if and only if
  $\inf_{\sph{d-1}} \Re(V)>0$.
\end{lemma}
From Theorem \ref{t:che}, if $\mu$ has a square-integrable spectral
spherical density $P$, then 
\begin{align} \label{e:V-sphere}
  V = \sumzi j a_j \proj_j P \quad
  \text{with}~a_j =
  \begin{cases}
    w_j(\alpha) & j~\text{even,}\\
    -\iunit\tan(\pi\alpha/2) w_j(\alpha)
    & \text{else.}
  \end{cases}
\end{align}
Moreover, the following is true.
\begin{cor} \label{c:che}
  Let $|z|^\alpha V(u_z)$ be the characteristic exponent of an
  $\alpha$-stable distribution $\mu$ with $\alpha\ne1$.  Then $\mu$
  has a square-integrable spectral spherical density if and
  only if $\sum j^{2\alpha+d} \Norm{\proj_j V}^2_{L^2(\sph{d-1})} <
  \infty$.
\end{cor}

Let $X\sim\mu$.  Given $u\in \sph{d-1}$, $Y=\ip u X$ is univariate
$\alpha$-stable.  Since its characteristic exponent $\Phi_X(t)$ is
equal to $\Phi_\mu(tu)$, from Theorem \ref{t:che},
\begin{align*} 
  \Phi_Y(t) =
  \begin{cases}
    \makebox[10.4cm][l]{$
      \displaystyle
      |t|^\alpha \Sbr{\sum_{j\mathrm{~even}}
        w_j(\alpha) P_j(u)
        -\iunit\tan\frac{\pi\alpha} 2\sgn t
        \sum_{j\mathrm{~odd}}
        w_j(\alpha) P_j(u)}$,}\quad\alpha\ne1\\
    \displaystyle
    |t|
    \Sbr{\sum_{j\mathrm{~even}} w_j(1) P_j(u)
      + \iunit \ln |t|\frac{2\pi^{d/2-1}}{\Gamma(d/2+1)} P_1(u)}
    \!+ \iunit t\frac2\pi \sum_{j\mathrm{~odd}}
    w^*_j P_j(u), \\
    \makebox[10.4cm]{}\quad\alpha=1
  \end{cases}
\end{align*}
For different $\alpha$, while $\Phi_Y$ is different, it always
satisfies the characterization in \cite{sato:99:cup}, Theorem 14.15.
It follows that the spherical harmonics of $P$ satisfy 
\begin{align} \label{e:proj-ineq}
  \sum_{j \rm{~even}} w_j(\alpha) P_j(u)  \ge
  \Abs{\sum_{j\rm{~odd}} w_j(\alpha) P_j(u)}
\end{align}
for all $\alpha\in (0,2)$ and $u\in \sph{d-1}$.  Equality can hold for
some $u\in \sph{d-1}$, in which case the \levy measure of $\ip u X$ is
concentrated on a half line according to the proof of Theorem 14.15 of
\cite{sato:99:cup}.  On the other hand, the result below holds.  For
$z\in\Coms$, denote by $\arg z$ the principal argument of $z$,
i.e., the unique $\theta\in(-\pi, \pi]$ with $z=|z| e^{\iunit
  \theta}$.
\begin{cor} \label{c:proj-ineq}
  Let $P\ne0$ be a polynomial.
  \begin{enumerate}[itemsep=0ex, parsep=0ex, leftmargin=3ex,
    topsep=.4ex, label=\arabic*)]
  \item \label{i:proj1}
    Strict inequality holds in \eqref{e:proj-ineq}.
  \item \label{i:proj2}
    If $\alpha\ne1$, then for $V$ in \eqref{e:che-ss3}, $\sup_{u\in
      \sph{d-1}} |\arg V(u)|< (\pi/2)\min(\alpha, 2-\alpha)$.
  \end{enumerate}
\end{cor}

\subsection{Proof of Theorem \ref{t:che}}
Recall that $\lambda = P\omega$ is a finite measure.  By assumption,
$P\in L^2(\sph{d-1})$.  Since $P_j = \proj_j P$, then $P = \sumzi j
P_j$ in $L^2(\sph{d-1})$.  Write
\eqref{e:che-ss} as
\begin{align} \label{e:che-ss2}
  \Phi_\mu(z)
  =
  \begin{cases}
    \displaystyle
    \sint |\ip z v|^\alpha\,\lambda(\dd v)
    -
    \iunit\tan\frac{\pi\alpha} 2
    \sint |\ip z v|^\alpha\,\sgn{\ip z v}\,
    \lambda(\dd v) & \alpha\ne1,
    \\[2ex]
    \displaystyle\sint
    |\ip z v|\,\lambda(\dd v) + \iunit \frac{2}{\pi}
    \sint \ip z v\ln|\ip z v|\,\lambda(\dd v)
    & \alpha=1.
  \end{cases}
\end{align}
Since $|\ip z v| = |z|\cdot |\ip {u_z} v|$ is bounded and symmetric in
$v$, by $P_j(-v) = (-1)^j P_j(v)$,
\begin{align*}
  \sint |\ip z v|^\alpha \lambda(\dd v)
  &=
    |z|^\alpha \sum_{j\text{~even}}
    \sint |\ip {u_z} v|^\alpha P_j(v)\,\omega(\dd v),\\
  \sint |\ip z v|^\alpha \sgn\ip z v\,\lambda(\dd v)
  &=
    |z|^\alpha \sum_{j\text{~odd}}
    \sint |\ip {u_z} v|^\alpha \sgn\ip{u_z} v\,
    P_j(v)\,\omega(\dd v).
\end{align*}
For $j\in\Ints_+$, let $\rx_j = \cf{j\text{~is odd}}$.  The next step
is to evaluate 
\begin{align*}
  \sint |\ip u v|^\alpha (\sgn\ip u v)^{\rx_j}
  P_j(v)\,\omega(\dd v).
\end{align*}

The following derivation applies to \emph{all\/}
$\alpha\in(0,\infty)$.  Since the mapping $t\mapsto |t|^\alpha
(\sgn t)^{\rx_j}$ is continuous, by Funk-Hecke formula \eqref {e:Funk-Hecke0}, for all $h\in \Cal H_{j,d}$ and $u\in  \sph{d-1}$,
\begin{gather} \label{e:Funk-Hecke}
  \sint |\ip u v|^\alpha (\sgn\ip uv)^{\rx_j} h(v)
  \omega(\dd v)  = w_j(\alpha) h(u),
  \intertext{where} \label{e:F-H-G}
  w_j(\alpha)
  =
  A(\sph{d-2})
  \int^1_{-1} |t|^\alpha (\sgn t)^{\rx_j}
  \wt C^{(d-2)/2}_j(t) (1-t^2)^{(d-3)/2}\,\dd t.
\end{gather}
Other than the exact values of $w_j(\alpha)$, \eqref{e:che-ne1}
follows immediately from \eqref{e:che-ss2} and \eqref{e:Funk-Hecke} if 
$\alpha\ne1$.  With similar consideration, it can be expected that
\eqref{e:che-eq1} holds as well.  Thus, the task is to show that
$w_j(\alpha)$ and $w^*_j$ are given by \eqref{e:Funk-Hecke2} and
\eqref{e:che-eq1d}, respectively.

From the polynomial expression of $\wt C^{(d-2)/2}_j(t)$, it is
possible to obtain a closed form of $w_j(\alpha)$ from \eqref
{e:F-H-G} by integration term by term.  However, this calculation does
not directly lead to the desired form of $w_j(\alpha)$.  Instead,
by \cite{andrews:99:cup}, Exercise 6.28, if $b>0$, then
\begin{align*} 
  \wt C^b_j(t)
  = t^{\rx_j} {_2}F_1\Grp{
  \begin{matrix}
    -\Flr{j/2}, \Flr{j/2}+\rx_j +b\\b+1/2
  \end{matrix}; 1-t^2
  }
\end{align*}
and by continuity, the identity still holds if $b=0$.  Then by
\eqref{e:F-H-G},
\begin{align}  \label{e:L-abj}
  w_j(\alpha) = A(\sph{d-2}) L_j(\alpha, (d-2)/2),
\end{align}
where for $a>0$, $b\ge 0$, and $j\in\Ints_+$,
\begin{align*}
  L_j(a,b)
  =
  \int^1_{-1} |t|^{a +\rx_j} {_2}F_1\Grp{
  \begin{matrix}
    -\Flr{j/2}, \Flr{j/2} +\rx_j +b\\b+1/2
  \end{matrix}; 1-t^2
  } (1-t^2)^{b-1/2}\,\dd t.
\end{align*}

To evaluate $L_j(a,b)$,  by change of variable $s =
1-t^2$,  \label{e:cv-L} 
\begin{align*}
  L_j(a,b)
  =
  \int^1_0 s^{b-1/2}\,(1-s)^{(a+\rx_j-1)/2}\,
  {_2}F_1\Grp{
  \begin{matrix}
    -\Flr{j/2}, \Flr{j/2}+\rx_j+b\\b+1/2
  \end{matrix}; s
  } \dd s.
\end{align*}
By \cite{andrews:99:cup}, Theorem 2.2.4, for $p_i$, $q_i\in\Coms$
with $\Re q_i>0$, $i=1,2$, and $x\in \Coms\setminus [1,\infty)$,
\begin{align*}
  \int^1_0 s^{q_1-1} (1-s)^{q_2-1} 
  {_2}F_1\Grp{
  \begin{matrix}
    p_1, p_2\\q_1
  \end{matrix}; xs
  }\,\dd s
  = B(q_1, q_2)\, {_2}F_1\Grp{
  \begin{matrix}
    p_1, p_2\\q_1+q_2
  \end{matrix}; x
  },
\end{align*}
where $B(q_1, q_2) = \lfrac{\Gamma(q_1)\Gamma(q_2)}{\Gamma(q_1 +
  q_2)}$ is the Beta function.  If $p_1$ or $p_2$ is a nonpositive
integer, then on each side of the display the ${_2}F_1$ function is a
polynomial of finite degree, and so the identity holds for all
$x\in\Coms$, in particular, for $x=1$.  Then
\begin{align*}
  L_j(a, b) = B(b+1/2, (a + \rx_j+1)/2)\,
  {_2}F_1\Grp{
  \begin{matrix}
    -\Flr{j/2}, \Flr{j/2}+\rx_j+b\\b+(a + \rx_j)/2+1
  \end{matrix}; 1
  }.
\end{align*}
By Gauss' formula \eqref{e:Gauss2F1}, 
\begin{align*}
  L_j(a,b)
  &=
    B(b+1/2, (a + \rx_j+1)/2)
  \\
  &\quad\times
    \frac{\Gamma(b+(a+\rx_j)/2+1)
    \Gamma((a-\rx_j)/2+1)}
    {\Gamma(b+(a+\rx_j)/2+1+ \Flr{j/2})
    \Gamma((a-\rx_j)/2+1-\Flr{j/2})}
  \\
  &=
    \frac{\Gamma(b+1/2)}{\Gamma(b+(a+j)/2+1)\Gamma((a-j)/2+1)}
    \times
    \Gamma((a+1)/2)\Gamma(a/2+1),
\end{align*}
the second equality due to $\Gamma((a + \rx_j+1)/2) \Gamma((a -
\rx_j)/2+1) = \Gamma((a+1)/2) \Gamma(a/2+1)$ and $\rx_j/2 + \Flr{j/2}
= j/2$.  Then by Legendre's duplication formula \eqref{e:L-dup},
\begin{align} \label{e:L-abj2}
  L_j(a,b)
  =
  \frac{\sqrt\pi\,\Gamma(b+1/2) \Gamma(a+1)}
  {2^a\Gamma(b+(a+j)/2+1)\Gamma((a-j)/2+1)}.
\end{align}
By \eqref{e:L-abj}, the first expression in \eqref{e:Funk-Hecke2}
follows.  Then by Euler's reflection formula \eqref{e:Euler}, the
second expression in \eqref{e:Funk-Hecke2} follows.  This proves
\eqref{e:che-ne1} for $\alpha\ne1$ and also yields the terms $w_j(1)
P_j(u_z)$ in \eqref{e:che-eq1} with $j$ being even for $\alpha=1$.

For $\alpha=1$, it only remains to evaluate the second integral on the
\rhs of \eqref{e:che-ss2}.  First,
\begin{align*}
  \sint \ip z v\ln|\ip z v| \lambda(\dd v)
  &=|z|\sint \ip {u_z} v(\ln|z| + \ln|\ip {u_z} v|)\lambda(\dd v)
  \\
  =|z|\ln|z|
  &\sint \ip{u_z} v \lambda(\dd v) +|z| \sint \ip{u_z} v
    \ln|\ip{u_z} v|\,\lambda(\dd v).
\end{align*}
Given $z\ne0$, $\ip{u_z} \cdot\in \Cal H_{1,d}$, so the first integral
on the \rhs equals $\sint \ip{u_z} v P_1(v)\,\omega(\dd v)$.  Then by
$\ip u v = |\ip u v|\sgn{\ip u v}$,  \label{e:eq1-uv}
\begin{gather*}
  \sint \ip z v\ln|\ip z v| \lambda(\dd v)
  =|z|\ln |z| w_1(1) P_1(u_z) + |z| I(u_z),
  \intertext{where}
  I(u)
  =
  \sint \ip u v \ln  |\ip u v| \,
  \lambda(\dd v)
  =
  \sum_{j\text{~odd}}
  \sint \ip u v \ln  |\ip u v| P_j(v)\, \omega(\dd v).
\end{gather*}

From the first expression in \eqref{e:Funk-Hecke2},
\begin{align}\label{e:che-eq1a} 
  w_1(1) = \frac{\pi^{d/2}}{\Gamma(d/2+1)}.
\end{align}
Following the proof for $\alpha\ne1$, for each odd $j\ge1$, there is a
constant $w^*_j$ such that
\begin{align*}
  \sint \ip u v \ln  |\ip u v| P_j(v)\,
  \omega(\dd v)
  = w^*_j P_j(u),
\end{align*}
which together with the last two displays yields
\begin{align}\label{e:che-eq1c}
  \sint \ip z v\ln|\ip z v| \lambda(\dd v)
  =
  |z|\ln |z| \frac{\pi^{d/2}}{\Gamma(d/2+1)} P_1(u_z) +
  |z| \sum_{j\text{~odd}} w^*_j P_j(u_z).
\end{align}
Then by \eqref{e:che-ss2}, \eqref{e:che-eq1} holds.  The task now is
to show \eqref{e:che-eq1d} for odd $j\ge1$.  Following the proof of
\eqref{e:L-abj}, $w^*_j = A(\sph{d-2}) \wt L_j(1, (d-2)/2)$, where for
$a>0$ and $b\ge0$, 
\begin{align*}
  \wt L_j(a,b) = 
  \int^1_{-1} |t|^{a+1} \ln|t|\,{_2}F_1\Grp{
  \begin{matrix}
    -\Flr{j/2}, \Flr{j/2} +1 +b\\b+1/2
  \end{matrix}; 1-t^2
  } (1-t^2)^{b-1/2}\,\dd t.
\end{align*}
By dominated convergence, $\wt L_j(a,b) = \lfrac{\partial
  L_j(a,b)}{\partial a}$.  Then by \eqref{e:L-abj2},
\begin{align} \nonumber
  \wt L_j(a,b)
  &=
    [- \ln 2 + \psi(a+1) - \psi(b+(a+j)/2+1)/2]
    L_j(a,b)
  \\\label{e:Ltilde}
  &\quad+ 
    \frac{\sqrt\pi \Gamma(b+1/2) \Gamma(a+1)}{2^{a+1}
    \Gamma(b+(a+j)/2+1)} g'((a-j)/2+1),
\end{align}
where $\psi(x) = \lfrac{\Gamma'(x)}{\Gamma(x)}$ and $g(x)
=\lfrac1{\Gamma(x)}$.  Let $a=1$ and $b=(d-2)/2$.  Multiply both sides
of \eqref{e:Ltilde} by $A(\sph{d-2}) = \frac{2\pi^{(d-1)/2}}{\Gamma((d
  -1)/2)}$.  Then
\begin{align*}
  w^*_j = [- \ln 2 + \psi(2) - \psi((d+1+j)/2)/2] w_j(1) +
  \frac{\pi^{d/2}}{2 \Gamma((d+1+j)/2)} g'((3-j)/2).
\end{align*}
From \cite{andrews:99:cup}, p.~13, for $n\in\Ints_+$, $\psi(n+1) =
-\gamma + s_n$, $\psi(n+1/2) = -\gamma - 2\ln 2 + 2 s_{2n} - s_n$,
where $\gamma$ is Euler's constant and $s_n = s_{n-1} + 1/n$ for
$n\ge1$ with $s_0=0$.  Then $\psi(2) = 1 - \gamma$ and $g'(1) =
-\psi(1)/\Gamma(1) = \gamma$, which together with \eqref{e:che-eq1a}
yields
\begin{align*}
  w^*_1 = \Sbr{-\ln 2 + 1- \frac{\gamma+\psi(d/2+1)}2}
  \frac{\pi^{d/2}}{\Gamma(d/2+1)}.
\end{align*}
The expression of $\psi(d/2+1)$ in terms of $s_n$'s then yields
$w^*_1$ in \eqref{e:che-eq1d}.  For $j>1$ odd, since $(j-1)/2$ is a
positive integer, from the first expression in \eqref{e:Funk-Hecke2},
$w_j(1)=0$.  This can also be derived using Rodrigues formula
(\cite{andrews:99:cup}, Eq.~(6.4.14)): for any $b$,
\begin{align*}
  C^b_j(t) (1-t^2)^{b-1/2}
  = \frac{(-2)^j (b)_j}{j!(j+2b)_j} [(1-t^2)^{b+j-1/2}]
  \Sp j.
\end{align*}
Then $w_j(1)=0$ for $d>2$ follows from \eqref{e:F-H-G} and integration
by parts.  The case $d=2$ can be shown by continuity argument or using
the property of Tchebyshev polynomial $\wt C^0_j(t)=T_j(t)$ (cf.\
\cite{andrews:99:cup}, p.~101).

On the other hand, $g(x) = \pi^{-1} \sin(\pi x)\Gamma(1-x)$ by Euler's
reflection formula \eqref {e:Euler}.  Then $g'(x) = \cos(\pi x)
\Gamma(1-x) - \pi^{-1} \sin(\pi x) \Gamma'(1-x)$ and so for
$n\in\Ints_+$, $g'(-n) = (-1)^n\Gamma(n+1)$, which together with
\eqref{e:che-eq1a} gives $w^*_j$ for odd $j>1$ in \eqref{e:che-eq1d}.
This shows \eqref{e:che-eq1} for $\alpha=1$. 
  
It only remains to show that \eqref{e:che-ne1} and \eqref
{e:che-eq1} are \ac.  It suffices to show
\begin{align} \label{e:wp}
  \sumzi j |w_j(\alpha)| \sup_{\sph{d-1}} |P_j|<\infty, \quad
  \sumzi j |w^*_j| \sup_{\sph{d-1}} |P_j|<\infty.
\end{align}

\begin{lemma} \label{l:sup-sp}
  For $j\ge1$ and $P\in \Cal H_{j,d}$,
  \begin{align*}
    \sup_{\sph{d-1}} |P|^2 \le \frac{c_{j,d}}{A(\sph{d-1})}
    \sint |P|^2\,\dd\omega,
  \end{align*}
  where $c_{j,d} = \dim(\Cal H_{j,d})$ is specified in
  \eqref{e:dim-h}.
\end{lemma}
\begin{proof}
  Let $\sigma^2 = \sint |P|^2\,\dd\omega$ and $\varphi_1 = P/\sigma$.
  Then $\varphi_1\in\Cal H_{j,d}$ and there are $\varphi_2$, \ldots,
  $\varphi_{c_{j,d}} \in \Cal H_{j,d}$, such that together with
  $\varphi_1$ they form an orthonormal basis of $\Cal H_{j,d}$.  Then
  by \cite {andrews:99:cup}, Theorem 9.6.3 or \cite{stein:71:pup},
  Corollary IV.2.9(b),
  \begin{align*}
    \sup_{\sph{d-1}} |P|^2
    =
    \sigma^2 \sup_{\sph{d-1}} |\varphi_1|^2
    \le
    \sigma^2 \sup_{\sph{d-1}} \sum^{c_{j,d}}_{i=1} |\varphi_i|^2
    =
    \sigma^2 \cdot\frac{c_{j,d}}{A(\sph{d-1})}.  \tag*{\qedhere}
  \end{align*}
\end{proof}
Let $\sigma^2_j= \sint |P_j|^2\,\dd\omega$.  Then $\sumzi j\sigma^2_j=
\sigma^2 = \sint |P|^2\,\dd\omega$.  By Cauchy--Schwarz inequality
and Lemma \ref{l:sup-sp},
\begin{align*}
  \Grp{\sumzi j |w_j(\alpha)| \sup_{\sph{d-1}}|P_j|}^2
  &\le \Sbr{\sumzi j \frac{\sup_{\sph{d-1}} |P_j|^2}{c_{j,d}}}
    \Sbr{\sumzi j c_{j,d} w_j(\alpha)^2}    
  \\
  &\le \Sbr{\sumzi j \frac{\sigma^2_j}{A(\sph{d-1})}}
    \Sbr{\sumzi j c_{j,d} w_j(\alpha)^2}
    =\frac{\sigma^2}{A(\sph{d-1})}\sumzi j c_{j,d} w_j(\alpha)^2.
\end{align*}
From the second expression in \eqref{e:Funk-Hecke2} and Stirling's
formula \eqref{e:Stirling}, as $j\toi$, $w_j(\alpha) =
O(j^{-\alpha-d/2})$.  Then by \eqref{e:c-asym}, $c_{j,d} w_j(\alpha)^2
= O(j^{-2-2\alpha})$, which implies the first half of \eqref{e:wp}.
Next, from \eqref {e:che-eq1d}, $w^*_j = O(j^{-1-d/2})$, so by similar
argument, the second half of \eqref{e:wp} follows.

\subsection{Proof of other results}
\begin{proof}[Proof of Corollary \ref{c:che}]
Let $\lambda$ be the spectral spherical measure of $\mu$.  First
show that $\lambda$ has a density in $L^2(\sph{d-1})$ if and only if
$\Delta=\sumzi j |a_j|^{-2} \Norm{\proj_jV}^2<\infty$, where
$\Norm\cdot$ denotes  $\Norm\cdot_{L^2(\sph{d-1})}$ for simplicity
and $a_j$ are defined in \eqref{e:V-sphere}.  Notice that since
$\alpha\ne1$, $a_j\ne0$ for all $j\in\Ints_+$.  Also, since $V$ is
bounded by Lemma \ref{l:ReV}, it is clearly in $L^2(\sph{d-1})$. 

Suppose $\lambda = P\omega$ with $P\in L^2(\sph{d-1})$.  Then by
\eqref{e:V-sphere}, $a^{-1}_j \proj_j V = \proj_j P$,
so $\Delta = \sumzi j \Norm{\proj_j P}^2 = \Norm P^2 < \infty$.
Conversely, suppose $\Delta < \infty$.  Then $\sumzi j a^{-1}_j
\proj_j V$ converges in $L^2(\sph{d-1})$.  Let $Q$ be the limit.
We need the following result, which is also useful later.
\begin{lemma} \label{l:Q}
  $Q$ is real-valued on $\sph{d-1}$.
\end{lemma}
\begin{proof}
  Let $S_j = (-\iunit)^j \proj_j (V)$, $j\in\Ints_+$.  From $V(-u) =
  \conj{V(u)}$, $S_j$ is real-valued.    Indeed, in general, suppose
  $f\in L^2(\sph{d-1})$ and define $f^*\in L^2(\sph{d-1})$ such that
  $f^*(u) = f(-u)$.   Then from \eqref{e:proj}, $\proj_j (f^*)
  = (-1)^j \proj_j f$ and $\conj{\proj_j f} = \proj_j \conj f$.  If
  $f^* = \conj f$, then $\conj{(-\iunit)^j \proj_j f} = \iunit^j
  \conj{\proj_j f} = \iunit^j \proj_j(f^*) = (-\iunit)^j \proj_j f$,
  so $(-\iunit)^j \proj_j f$ is real-valued.  Now $Q = \sum (\iunit^j
  a^{-1}_j) S_j$.  From \eqref{e:V-sphere}, $\iunit^j a^{-1}_j$ is
  real-valued for all $j\in\Ints_+$.  Then $Q$ is real-valued.
\end{proof}
  
Continuing the proof of Corollary \ref{c:che}, since by
Cauchy--Schwarz inequality $\sint |Q|\,\dd\omega\le \{A(\sph{d-1})
\sint |Q|^2\,\dd\omega\}^{1/2} < \infty$, $\wt \lambda= Q\,\omega$ is
a finite (signed) measure on $\sph{d-1}$.  Denote by $\nu$ the \levy
measure of $\mu$.  From Remark 14.6 and the proof of Theorem 14.10 of \cite{sato:99:cup},
\begin{align} \label{e:Levy}
  \nu(B) =
  \frac{1}{C_\alpha}
  \sint \lambda(\dd v) \int_{r>0} \cf{r v\in B}
  r^{-1-\alpha} \,\dd r,
\end{align}
where $C_\alpha = (\pi/2)[\sin(\pi\alpha/2)\Gamma(1+\alpha)]^{-1}$.
Define measure $\wt\nu(B)$ similarly, but with $\wt\lambda$ replacing
$\lambda$.  Let $\wt\lambda_+ = Q_+\omega$ and $\wt\lambda_- =
Q_-\omega$, where for $c\in\Reals$, $c_\pm = \max(0, \pm c)$.
Define measures $\wt\nu_+$ and $\wt\nu_-$ in terms of $\wt\lambda_+$
and $\wt\lambda_-$, respectively.  Then $\wt\nu = \wt\nu_+ -
\wt\nu_-$.  Since $\wt\nu_+(\{0\})=0$ and
\begin{align*}
  \rint \min(|x|^2,1)\,\wt\nu_+(\dd x)
  \le
  \frac{1}{C_\alpha}
  \sint |Q|\,\omega(\dd v) \int_{r>0} \frac{\min(r^2,1)}{r^{1+\alpha}}
  \,\dd r < \infty,
\end{align*}
$\wt\nu_+$ is a \levy measure.  Likewise, $\wt\nu_-$ is a \levy
measure.

Let $0<\alpha<1$.  Similar to the above display, $\rint\min(|x|, 1)\,
\wt\nu_\pm(\dd x)<\infty$.   Then the argument in the proof of Theorem
14.10 in \cite{sato:99:cup} can be applied to get
\begin{align*}
  \rint (e^{\iunit\ip z x}-1) \wt\nu(\dd x)
  &=
    \nth{C_\alpha}
    \sint \wt\lambda(\dd v) \int_{r>0}(e^{\iunit |z|\ip {u_z} v}-1)
    r^{-1-\alpha} \,\dd r\\
  &=
    -|z|^\alpha \sint |\ip {u_z}  v|^\alpha[1 - \iunit
    \tan(\lfrac{\pi\alpha}{2}) \sgn(\ip{u_z} v)]\, \wt\lambda(\dd v),
\end{align*}
where $C_\alpha$ is as in \eqref{e:Levy}.  The argument for Theorem \ref{t:che} can be directly applied to $\wt\lambda$.  Thus,
\begin{align*}
  \rint (e^{\iunit\ip z x}-1)\, \wt\nu(\dd x)
  =
  -|z|^\alpha \sumzi j a_j (\proj_j Q)(u_z).
\end{align*}
Then by $a_j \proj_j Q = a_j \proj_j (\sum a^{-1}_i \proj_i V) = 
\proj_j V$ in $L^2(\sph{d-1})$,
\begin{align*}
  \rint (e^{\iunit\ip z x}-1)\, \wt\nu(\dd x)
  =
  -|z|^\alpha \sumzi j (\proj_j V)(u_z)
  =
  -|z|^\alpha V(u_z) = \rint (e^{\iunit \ip z x} -1)
  \,\nu(\dd x).
\end{align*}
Note that in general the first two equalities only hold for a.e.\ $z$.
However, by continuity in $z$, the integrals on both ends equal for
all $z$.  It follows that
\begin{align*}
  \rint (e^{\iunit\ip z x}-1)\, \wt\nu_+(\dd x)
  = \rint (e^{\iunit \ip z x} -1)\,(\nu+\wt\nu_-)(\dd x).
\end{align*}
From the proof of uniqueness of the \levy measure for an infinitely
divisible distribution (cf.\ the one for Theorem 8.1(ii) in \cite
{sato:99:cup}), $\wt\nu_+ = \nu + \wt\nu_-$.  Then $\nu = \wt\nu$.  As
a result, $\lambda =\wt \lambda = Q\omega$, so $Q$ is the spectral
spherical density of $\nu$.  The case $1<\alpha<2$ can be similarly
proved.
 
Finally, from \eqref{e:V-sphere}, as $j\toi$, $a_j\asymp w_j(\alpha)$,
i.e., they have the same order.  From the second expression in
\eqref{e:Funk-Hecke2} $w_j(\alpha)\asymp \Gamma(\frac{j - \alpha} 2)/
\Gamma(\frac{j + \alpha+d}2)$.  Then by \eqref{e:Stirling}, $a_j\asymp
j^{-(\alpha+d/2)}$.  As a result, $\Delta<\infty \Iff \sum
j^{2\alpha+d} \Norm{\proj_j V}^2 < \infty$, completing the proof.
\end{proof}

\begin{proof}[Proof of Corollary \ref{c:proj-ineq}]
  \ref{i:proj1}
  Fix $\alpha\in (0,2)$.  If $X\sim\mu$ has \levy measure $\nu$,
  then for $u\in\sph{d-1}$, $\ip u X$ has \levy measure $\nu_1(B)
  = \nu(\{x{:}~\ip u x\in B\})$.  If equality holds for $u$ in
  \eqref{e:proj-ineq}, then $\nu_1$ is concentrated on a half line, so
  that, say $\nu_1((0,\infty)) = \nu(\{x{:}~\ip u x>0\})=0$.  Then
  from \eqref{e:Levy}, $\lambda(\{v\in \sph{d-1}{:}~\ip u v>0\})=0$,
  giving $P(v)=0$ for all $v\in \sph{d-1}$ with $\ip u v>0$.  Since
  $P$ is a polynomial, this implies $P\equiv0$, a contradiction.

  \ref{i:proj2} From Lemma \ref{l:ReV} and \ref{i:proj1},
  \begin{align*}
    \arg V=\arg[1 - \iunit r(V)
    \tan(\lfrac{\pi\alpha}{2})]
    =
    \begin{cases}
      -\arctan[r(V)\tan \frac{\pi\alpha}{2}] & \text{if~}
      \alpha\in (0,1)
      \\
      \arctan[r(V) \tan(\pi - \frac{\pi\alpha}{2})] & \text{if~}
      \alpha\in (1,2)
    \end{cases}
  \end{align*}
  where $r(V)=\Im(V)/\Re(V)$ is continuous on $\sph{d-1}$ with maximum
  absolute value strictly less than 1.  Then the claim follows.
\end{proof}

\section{Density when $\alpha\in (0,1)$}  \label{s:pdf(0,1)}
In this section, let $\mu$ be a nondegenerate $\alpha$-stable
distribution with $\alpha\in (0,1)$.  Without loss of generality,
assume $\mu$ has no shift, so that its characteristic exponent is
$\Phi_\mu(z) = |z|^\alpha V(u_z)$ as in \eqref{e:che-ss3}.  From
Lemma \ref{l:ReV}, $V^p$ is bounded for any $p\in\Reals$, where the
principal branch of the power function is used when $p$ is a
noninteger.  Denote
\begin{align} \label{e:Sjp}
  S_{j,p} = (-\iunit)^j \proj_j (V^p), \quad j\in\Ints_+.
\end{align}
Then $S_{j,p}\in\Cal H_{j,d}$ and
\begin{align}\label{e:V-SH}
  V^p
  = \sumzi j \iunit^j S_{j,p} \text{~in $L^2(\sph{d-1})$}.
\end{align}
Since $[V(-u)]^p = \conj{[V(u)]^p}$, from the proof of Lemma
\ref{l:Q}, it is seen that $S_{j,p}$ is real-valued.

\begin{theorem} \label{t:pdf-lt1}
  Let $\mu$ be an $\alpha$-stable distribution on $\Reals^d$, $d\ge2$,
  with $\alpha\in(0,1)$ and characteristic exponent \eqref{e:che-ss3}.
  If $\lambda = P\omega$ with $0\ne P\in \Cal P_{\dg,d}$ for some
  $\dg\in\Ints_+$, then $\FT\mu\in L^1(\Reals^d)$ and the \pdf of
  $\mu$ is
  \begin{align} \label{e:mvs-pdf<1}
    g(x)
    =
    \sumoi k \frac{(-2^\alpha)^k \pi^{-d/2}}{k! |x|^{k\alpha+d}}
    \sum^{k\dg}_{j=0}
    \frac{\Gamma((j+k\alpha+d)/2)}{\Gamma((j-k\alpha)/2)}
    S_{j,k}(u_x), \quad x\ne0.
  \end{align}
  The series is uniformly \ac in $\{x{:}~|x|\ge r\}$ for any $r>0$.
\end{theorem}
\begin{proof}[Remark]{\ }
  \begin{remarklist}
  \item By \eqref{e:V-sphere}, if $P$ is a polynomial of degree
    $\dg$, i.e., $P\in \Cal P_{\dg, d}$, then for each $k\in\Ints_+$,
    $V^k\in\Cal P_{k\dg, d}$, so from Section \ref{s:prelim},
    $S_{j,k}=0$ for $j>k\dg$, and each $S_{j,k}$ with $j\le k\dg$ has
    a closed form.  As a result, all the terms in the series
    \eqref{e:mvs-pdf<1} have closed form expressions.

  \item If $V$ is a constant $c>0$, then $S_{j,k}=c^k$ for $j=0$ and
    $0$ for $j>0$.  Consequently, \eqref{e:mvs-pdf<1} yields the well
    known result for the spherically symmetric case
    (\cite{uchaikin:99:vsp}, Eq.~(7.5.6)).
    
  \item It would be more satisfactory if an \ac series similar to
    \eqref {e:mvs-pdf<1} could be obtained for $P\in L^2(\sph{d-1})$.
    Since given $k$, $\Gamma((j+k\alpha+d)/2) / \Gamma((j-k\alpha)/2)$
    grows in the same order as $j^{k\alpha}$, the desired absolute
    convergence would require that $S_{j,k}(u) =(\proj_j V^k)(u)$
    vanish rapidly as $j\toi$.  However, it is unclear whether this is
    actually the case.  \qedhere
  \end{remarklist}
\end{proof}

In general, the tail asymptotic behavior of a multivariate stable
distribution is quite complicated, depending on the fractal dimensions
on the spectral spherical measure \cite{watanabe:07:tams}.  However,
\cite {rvaceva:62:stmsp}, Theorem 4.2 suggests that under the
condition of Theorem \ref {t:pdf-lt1}, as $|x|\toi$, $g(x)$ should
behave as $C P(u_x)/|x|^{\alpha+d}$ for some constant $C>0$ .  Indeed,
from \eqref {e:mvs-pdf<1}, 
\begin{align*}
  g(x) = -\frac{2^\alpha \pi^{-d/2}}{|x|^{\alpha+d}}
  \sum^\dg_{j=0} \frac{\Gamma((j+\alpha+d)/2)}{\Gamma((j-\alpha)/2)}
  S_{j,1}(u_x) + O(|x|^{-2\alpha-d}).
\end{align*}
On the other hand, from \eqref{e:che-ne1} and \eqref{e:V-sphere},
\begin{align*}
  S_{j,1} = (-\iunit)^j a_j P_j = (-1)^{(j+\rx_j)/2}
  [\tan(\pi\alpha/2)]^{\rx_j} w_j(\alpha) P_j,
\end{align*}
where $\rx_j = \cf{j\text{~is odd}}$ and $P_j = \proj_j P$.  Then by
the second expression in \eqref {e:Funk-Hecke2},
\begin{align*}
  \frac{\Gamma((j+\alpha+d)/2)}{\Gamma((j-\alpha)/2)}  S_{j,1}
  &=2^{1-\alpha} \pi^{d/2-1}\Gamma(\alpha+1)
    (-1)^{(j+\rx_j)/2} [\tan(\pi\alpha/2)]^{\rx_j}
    \sin((j-\alpha)\pi/2) P_j
  \\
  &=-2^{1-\alpha} \pi^{d/2-1} \Gamma(\alpha+1) \sin(\pi\alpha/2) P_j.
\end{align*}
It follows that
\begin{align} \label{e:pdf-asym-lt1}
  g(x) = \frac{2\alpha\Gamma(\alpha) \sin(\pi\alpha/2)}\pi
  \cdot \frac{P(u_x)}{|x|^{\alpha+d}} + O(|x|^{-2\alpha-d}).
\end{align}

For the general case, where $\mu$ may not have a spectral spherical
density, by using some of the arguments in the proof for Theorem
\ref{t:pdf-lt1}, the following result can be proved.
\begin{prop} \label{p:pdf-lt1}
  Let $\mu$ be a nondegenerate $\alpha$-stable distribution on
  $\Reals^d$, $d\ge2$, with $\alpha\in (0,1)$ and characteristic
  exponent \eqref{e:che-ss3}.  Let the \pdf of $\mu$ be $g$.  For
  $x\ne0$, put
  \begin{align} \label{e:proj-g}
    \phi_j(x) = 
    \sumoi k \frac{(-2^\alpha)^k \pi^{-d/2}}{k! |x|^{k\alpha+d}}
    \frac{\Gamma((j+k\alpha+d)/2)}{\Gamma((j-k\alpha)/2)}
    S_{j,k}(u_x), \quad j\in\Ints_+.
  \end{align}
  \begin{remarklist}
  \item \label{p:pdf-lt1a}
    The series in \eqref{e:proj-g} is uniformly \ac in $\{x{:}~|x|\ge
    c\}$ for any $c>0$.
  \item \label{p:pdf-lt1b}
    Given $r>0$, $\phi_j(r\,\cdot)\in \Cal H_{j,d}$ and
    $g(r\,\cdot) = \sumzi j \phi_j(r\,\cdot)$ in $L^2(\sph{d-1})$,
    that is, $\proj_j g(r\,\cdot)=  \phi_j(r\,\cdot)$. 
  \item \label{p:pdf-lt1c}
    The \pdf of $|X|$ with $X\sim\mu$ is equal to
    \begin{align} \label{e:pdf-abs-lt1}
      g_{|X|}(r)
      =
      \sumoi k \frac{(-2^\alpha)^k \pi^{-d/2}}{k! r^{k\alpha+1}}
      \frac{\Gamma((k\alpha+d)/2)}{\Gamma(-k\alpha/2)}
      \sint V^k\,\dd\omega.
    \end{align}
  \end{remarklist}
\end{prop}

Combining the argument that leads to  \eqref{e:pdf-asym-lt1} and 
Proposition \ref{p:pdf-lt1}, it can be seen that if $\mu$ has a
square-integrable spectral spherical density $P$, then for each
$j\in\Ints$,
\begin{align*}
  (\proj_j g)(x) =
  \frac{2\alpha\Gamma(\alpha) \sin(\pi\alpha/2)}\pi
  \cdot \frac{(\proj_j P)(u_x)}{|x|^{\alpha+d}} + O(|x|^{-2\alpha-d})
\end{align*}
as $|x|\toi$.  This suggests the tail asymptotic in
\eqref{e:pdf-asym-lt1} may hold, which would be a more satisfactory
result.  Finally, consider the asymptotic expansion of $g(x)$ as
$x\to0$.  The treatment is simpler than the previous results and
follows the one for the univariate case (cf.\ \cite {uchaikin:99:vsp},
Section~4.2).
\begin{prop} \label{p:asym-lt1}
  Let $\mu$ be a nondegenerate $\alpha$-stable distribution on
  $\Reals^d$, $d\ge2$, with $\alpha\in (0,1)$ and characteristic
  exponent \eqref{e:che-ss3}.  Then given $n\in\Ints_+$, as $x\to0$,
  \begin{align} \label{e:asym-lt1}
    g(x) =
    \frac{2^{1-d}}{\pi^{d/2} \alpha}
    \sum^n_{k=0} \Gamma((k+d)/\alpha) \Grp{\frac{|x|}{2}}^k
    \sum^{\Flr{k/2}}_{m=0}
    \frac{(-1)^m S_{k-2m,-(k+d)/\alpha}(u_x)}{m!\Gamma(k-m+d/2)}
    + O(|x|^{n+1}).
  \end{align}
  Furthermore,
  \begin{align*}
    g(0) =
    \frac{\Gamma(d/\alpha)}{\alpha (2\pi)^d} 
    \int V^{-d/\alpha}\,\dd\omega.
  \end{align*}
\end{prop}

\subsection{Proof of Theorem \ref{t:pdf-lt1}}
By Lemma \ref{l:ReV}, $\FT \mu(z) = e^{-|z|^\alpha V(u_z)}$ is
integrable.  Then by dominated convergence, $g(x) =  \lim_{\rx\dto0}
g_\rx(x)$, where
\begin{align*}
  g_\rx(x)
  = 
  (2\pi)^{-d} \rint
  e^{-\iunit\ip x z} e^{-\rx |z| - |z|^\alpha V(u_z)}\,\dd z.
\end{align*}
Since $\alpha<1$ and $V$ is bounded on $\sph{d-1}$, $e^{-\rx
  |z| + |z|^\alpha |V(u_z)|}$ is integrable.  Then
\begin{align*}
  g_\rx(x)
  =
  (2\pi)^{-d} \sumzi k \frac{(-1)^k}{k!}
  \rint e^{-\iunit\ip x z}
  e^{-\rx |z|}|z|^{k\alpha} [V(u_z)]^k \,\dd z.
\end{align*}

Put $b=(d-2)/2$.  Given $\rx>0$, for each $k\ge0$,  \label{e:pdf-lt1-L2}
\begin{align*}
  &
    \rint e^{-\iunit\ip x z}  e^{-\rx |z|}|z|^{k\alpha}
    [V(u_z)]^k \,\dd z
  \\
  &=
    \intzi e^{-\rx s} s^{k\alpha+d-1} \Cbr{
    \sint e^{-\iunit |x|s\ip {u_x} v} [V(v)]^k\,\omega(\dd v)}\,\dd s
  \\
  &=
    \intzi e^{-\rx s} s^{k\alpha+d-1} \Cbr{
    \frac{(2\pi)^{d/2}}{(|x|s)^b}
    \sumzi j (-\iunit)^j J_{j+b}(|x| s) (\proj_j V^k)(u_x)}\,\dd
    s
  \\
  &=
    \intzi e^{-\rx s} s^{k\alpha+d/2} \Cbr{
    \frac{(2\pi)^{d/2}}{|x|^b}
    \sum^{k\dg}_{j=0} J_{j+b}(|x| s) S_{j,k}(u_x)}\,\dd
    s,
\end{align*}
where the second equality follows from \eqref{e:ft-j} and the last one
from $(-\iunit)^j \proj_j V^k = S_{j,k}\equiv0$ for $j>k\dg$.  Then
\begin{align*}
  \rint e^{-\iunit\ip x z}  e^{-\rx |z|}|z|^{k\alpha}
  [V(u_z)]^k \,\dd z
  =
  \sum^{k\dg}_{j=0} F_{j,k,\rx}(|x|)S_{j,k}(u_x),
\end{align*}
where
\begin{align*}
  F_{j,k,\rx}(|x|)
  &=
    (2\pi)^{d/2} |x|^{-b}
    \intzi e^{-\rx s} s^{k\alpha+d/2} J_{j+b}(|x| s) \,\dd s
  \\
  &=
    \frac{(2\pi)^{d/2}}{|x|^{k\alpha+d}}
    \intzi e^{-\delta s} s^{k\alpha+d/2} J_{j+b}(s)
    \,\dd s 
\end{align*}
with $\delta = \delta(x) = \lfrac\rx{|x|}$.  As a result,
\begin{align}  \label{e:I_r}
  g_\rx(x)
  =
  (2\pi)^{-d} \sumzi k \frac{(-1)^k}{k!} \sum^{k\dg}_{j=0}
  F_{j,k,\rx}(|x|) S_{j,k}(u_x).
\end{align}

Given $x\ne0$, let $\rx\dto0$.  Then $\delta\dto0$ and by \cite  {andrews:99:cup}, Eq.~(9.10.5),
\begin{align*}
  F_{j,k,\rx}(|x|)\to
  \frac{(2\pi)^{d/2}}{|x|^{k\alpha+d}}
  \frac{2^{k\alpha+d/2} \Gamma((j+k\alpha+d)/2)}{
  \Gamma((j-k\alpha)/2)}.
\end{align*}
Note that when $j=k\alpha$, the \rhs is zero.  Thus, once it is shown
that on the \rhs of \eqref{e:I_r}, the order of
the infinite sum over $k$ and the limit can interchange, then \eqref
{e:mvs-pdf<1} follows.

The rest of the proof except for its very end is devoted to the
proof of the interchangeability.  For $a\ge0$ and $c>0$, by \cite{andrews:99:cup},
Theorem 4.11.3,
\begin{align*}
  \intzi e^{-\delta s} J_a(s) s^{c-1} \,\dd s
  =
  \frac{\Gamma(a+c)}{2^a \delta^{a+c}
  \Gamma(a+1)} \,{_2}F_1\Grp{
  \begin{matrix}
    (a+c)/2, (a+c+1)/2\\a+1
  \end{matrix}; -\nth{\delta^2}
  }.
\end{align*}
Then by Pfaff identity (\cite{andrews:99:cup}, Theorem 2.2.5), the
\rhs is equal to
\begin{align*}
  &
    \frac{\Gamma(a+c)}{2^a \delta^{a+c}
    \Gamma(a+1)} \Grp{1+\nth{\delta^2}}^{-(a+c)/2}
    {_2}F_1\Grp{ 
    \begin{matrix} (a+c)/2,(a-c+1)/2\\a+1
    \end{matrix}; \nth{1+\delta^2}}
  \\
  &=
    \frac{\Gamma(a+c)}{2^a
    \Gamma(a+1)} \Grp{1+\delta^2}^{-(a+c)/2}
    {_2}F_1\Grp{ 
    \begin{matrix} (a+c)/2,(a-c+1)/2\\a+1
    \end{matrix}; \nth{1+\delta^2}}.
\end{align*}
Then, with $a = j+b$ and $c = k\alpha+d/2+1$,
\begin{align*}
  F_{j,k,\rx}(|x|)
  &=
    \frac{(2\pi)^{d/2}}{|x|^{k\alpha+d}}
    \frac{\Grp{1+\delta^2}^{-(j+k\alpha+d)/2}}{2^{j-1+d/2}}
    L_{jk}\Grp{\nth{1+\delta^2}},
\end{align*}
where
\begin{align*}
  L_{jk}(x)
  &=
    \frac{\Gamma(j+k\alpha+d)}{\Gamma(j+d/2)}
    \,{_2}F_1\Grp{ 
    \begin{matrix} (j+k\alpha+d)/2,(j-k\alpha-1)/2\\j+d/2
    \end{matrix}; x}
  \\
  &=
    \frac{\Gamma(j+k\alpha+d)}{\Gamma(j+d/2)}
    \sumzi m \frac{((j+k\alpha+d)/2)_m ((j-k\alpha-1)/2)_m}{
    m! (j+d/2)_m} x^m.
\end{align*}
As a result, by \eqref{e:I_r}
\begin{align}\label{e:I-L}
  g_\rx(x)
  =
  \sumzi k \frac{(-1)^k (2\pi)^{-d/2}}{|x|^{k\alpha+d} k!}
  \sum^{k\dg}_{j=0}
  \frac{\Grp{1+\delta^2}^{-(j+k\alpha+d)/2}}{2^{j-1+d/2}}
  L_{jk}\Grp{\nth{1+\delta^2}} S_{j,k}(u_x).
\end{align}

To prove the interchangeability of the infinite sum over $k$ and the
limit as $\rx\to0$, it suffices to consider the sum over $k\gg 1$.
For $j=0,\ldots, k\dg$, if $j\ge k\alpha+1$, then every coefficient in
the series expression of $L_{jk}(x)$ is nonnegative, so for $|x|\le
1$, $|L_{jk}(x)|\le L_{jk}(1)$.  By Gauss' formula \eqref{e:Gauss2F1}
followed by Legendre's duplication formula
\eqref{e:L-dup},  
\begin{align}\nonumber
  L_{jk}(1)
  &=
    \frac{\Gamma(j+k\alpha+d)\Gamma(1/2)}{
    \Gamma((j-k\alpha)/2) \Gamma((j+k\alpha+1+d)/2)}
  \\\label{e:Ljk(1)}
  &=
    \frac{2^{j+k\alpha+d-1}\Gamma((j+k\alpha+d)/2)}{
    \Gamma((j-k\alpha)/2)}.
\end{align}
Since $\Gamma(x)$ is log-convex on $x>0$, and since $j\le k\dg$, from
the above display,
\begin{align*}
  L_{jk}(1) \le \frac{2^{k(\alpha+\dg)+d-1}
  \Gamma((k\dg + k\alpha+d)/2)} {\Gamma((k\dg - k\alpha)/2)}.
\end{align*}
Since $j\le k\dg$, by assuming $j\ge k\alpha+1$, $\dg>\alpha$.  Apply
Stirling's formula \eqref{e:Stirling} to get
\begin{align*}
  L_{jk}(1)
  &\le \const^k
    \Grp{\frac{k\dg + k\alpha+d}{2e}}^{(k\dg + k\alpha+d)/2}
    \!\!\!\bigg/\!
    \Grp{\frac{k\dg - k\alpha}{2e}}^{(k\dg - k\alpha)/2}
  \\
  &\le\const^k\times
    \frac{k^{(k\dg + k\alpha+d)/2}}
    {k^{(k\dg - k\alpha)/2}} = \const^k\times k^{k\alpha + d/2}
    \le \const^k\times \Gamma(k\alpha),
\end{align*}
where $\const$ denotes a constant independent of $(k,j)$.  It
follows that
\begin{align}
  \label{e:Ljk-G}
  |L_{jk}(x)|
  \le \const^k\times \Gamma(k\alpha).
\end{align}

On the other hand, if $j<k\alpha+1$, then for $|x|\le 1$,
\begin{align} \label{e:Ljk-a}
  |L_{jk}(x)|
  \le
  \frac{\Gamma(j+k\alpha+d)}{\Gamma(j+d/2)}
  \sumzi m b_m
\end{align}
with
\begin{align*}
  b_m = 
  \frac{((j+k\alpha+d)/2)_m |((j-k\alpha-1)/2)_m|}{
  m! (j+d/2)_m}.
\end{align*}
Let $m_{jk} = \Cil{(k\alpha+1-j)/2}$.   Then for $m\ge m_{jk}$,
$(j-k\alpha-1)/2 + m\ge0$, so $((j-k\alpha-1)/2)_m$ has the same sign
as $((j-k\alpha-1)/2)_{m_{jk}}$.  Then
\begin{align*}
  \sumzi m b_m
  &=
  \sum_{m<m_{jk}} b_m
  +
  \Abs{\sum_{m\ge m_{jk}} \frac{((j+k\alpha+d)/2)_m
    ((j-k\alpha-1)/2)_m}{m! (j+d/2)_m}}
  \\
  &\le
    2\sum_{m<m_{jk}} b_m + 
    \Abs{{_2}F_1\Grp{ 
    \begin{matrix} (j+k\alpha+d)/2,(j-k\alpha-1)/2\\j+d/2
    \end{matrix}; 1} 
  }.
\end{align*}
By Gauss' formula \eqref{e:Gauss2F1} and Euler's reflection
formula \eqref{e:Euler}, 
\begin{align*}
  {_2}F_1\Grp{ 
  \begin{matrix} \frac{j+k\alpha+d}2, \frac{j-k\alpha-1}2\\j+ \frac d2
  \end{matrix}; 1}
  &=
    \frac{\Gamma(j+d/2) \sqrt\pi}{\Gamma((j-k\alpha)/2)
    \Gamma((j+k\alpha+1+d)/2)}
    \\
  &=
  \frac{\Gamma(j+d/2) \Gamma(1+(k\alpha-j)/2)
  \sin((j-k\alpha)\pi/2)}{\sqrt\pi \Gamma((j+k\alpha+1+d)/2)}.
\end{align*}
In particular, when $j=k=0$, ${_2} F_1(d/2, -1/2; d/2; 1)=0$.  Then
\begin{align}\label{e:a-bound}
  \sumzi m b_m
  &\le
    2\sum_{m<m_{jk}} b_m
    +
    \frac{\Gamma(j+d/2) \Gamma(1+(k\alpha-j)/2)}{
    \Gamma((j+k\alpha+1+d)/2)}.
\end{align}
For $0\le m<m_{jk}$, $(j-k\alpha-1)/2\le (j-k\alpha-1)/2+m<0$, so
$|(j-k\alpha-1)/2 + m| < (k\alpha+1-j)/2$.  As a result,
$|((j-k\alpha-1)/2)_m| \le [(k\alpha+1-j)/2]^m$.  On the other hand,
\begin{align*}
  \frac{((j+k\alpha+d)/2)_m}{(j+d/2)_m}
  \le 
  \frac{((j+k\alpha+d+1)/2)_m}{(j+d/2)_m}
\end{align*}
and by $j<k\alpha+1$, the \rhs is increasing in $m$.  Therefore,
\begin{align*}
  \sum_{m<m_{jk}} b_m
  &\le
    \frac{((j+k\alpha+d+1)/2)_{m_{jk}}}{(j+d/2)_{m_{jk}}}
    \sum_{m<m_{jk}} \frac{[(k\alpha+1-j)/2]^m}{m!}
  \\
  &\le 
    \frac{((j+k\alpha+d+1)/2)_{m_{jk}}}{(j+d/2)_{m_{jk}}}
    e^{(k\alpha+1)/2}
    \le
    \const^k\times
    \frac{((j+k\alpha+d)/2)_{m_{jk}}}{(j+d/2)_{m_{jk}}}.
\end{align*}
Combine this with \eqref{e:Ljk-a} and \eqref{e:a-bound}, and then 
plug in $m_{jk} = \Cil{(k\alpha+1-j)/2}$ to get \label{e:Ljk(x)}
\begin{align*}
  |L_{jk}(x)|
  &\le
    \const^k\times
    \frac{\Gamma(j+k\alpha+d)}{\Gamma((j+k\alpha+d)/2)}
    \frac{\Gamma((j+k\alpha+d)/2 + m_{jk})}{\Gamma(j+d/2+m_{jk})}
  \\
  &\quad
    + \frac{\Gamma(j+k\alpha+d)\Gamma(1+ (k\alpha-j)/2)}{
    \Gamma((j+k\alpha+1+d)/2)}
  \\
  &\le
    \const^k\times
    \frac{\Gamma(j+k\alpha+d)}{\Gamma((j+k\alpha+d)/2)}
    \frac{\Gamma(k\alpha+ (d+3)/2)}{\Gamma((j+k\alpha+d+1)/2)}
  \\
  &\quad
    + \frac{\Gamma(j+k\alpha+d)\Gamma(1+(k\alpha-j)/2)}{
    \Gamma((j+k\alpha+1+d)/2)}
  \\
  &\le
    \const^k\times
    [\Gamma(k\alpha+(d+3)/2)+\Gamma((j+k\alpha+d)/2)
    \Gamma(1+(k\alpha - j)/2)],
\end{align*}
where Legendre's duplication formula \eqref{e:L-dup} is used in the
last line.  Since $k\gg1$, from \eqref{e:Stirling},
\begin{align*}
  |L_{jk}(x)|
  \le
  \const^k\times [\Gamma(k\alpha) +
  \Gamma(1+(j+k\alpha)/2)\Gamma(1+(k\alpha-j)/2)].
\end{align*}
Recall $0\le j<k\alpha+1$.  By the log-convexity of $\Gamma(x)$ on
$x>0$, the \rhs is no greater than the value when $j=k\alpha+1$.  Then
\eqref{e:Ljk-G} again holds.

As a result, for $k\gg1$, the $k\th$ summand in \eqref{e:I-L} has
absolute value no greater than
\begin{align} \label{e:I-term-bound0}
  \frac{\const^k \times \Gamma(k\alpha)}{|x|^{k\alpha+d} k!}
  \sum^{k\dg}_{j=0} |S_{j,k}(u_x)|.
\end{align}
Since $S_{j,k} = (-\iunit)^j \proj_j(V^k)\in\Cal H_{j,d}$, by Lemma
\ref {l:sup-sp},
\begin{align*} 
  \sup_{\sph{d-1}} |S_{j,k}|^2
  \le  \frac{c_{j,d}}{A(\sph{d-1})}\|S_{j,k}\|^2_{L^2(\sph{d-1})}
  \le \frac{c_{j,d}}{A(\sph{d-1})}\|V^k\|^2_{L^2(\sph{d-1})}
  \le c_{j,d}\sup_{\sph{d-1}} |V^{2k}|.
\end{align*}
Then by \eqref{e:c-asym} and Lemma \ref{l:ReV}
\begin{align} \label{e:Sjk-bound}
  \sup_{\sph{d-1}} |S_{j,k}| \le  \sqrt{c_{j,d}} \sup_{\sph{d-1}} |V|^k
  \le \const\times (j+1)^{d/2-1}
  \sup_{\sph{d-1}} |V|^k.
\end{align}
It follows that \eqref{e:I-term-bound0} is bounded by $D_k = \const^k
\times\Gamma(k\alpha)/(|x|^{k\alpha+d} k!)$.  Since $\alpha \in
(0,1)$, by Stirling's formula \eqref{e:Stirling}, $\sumzi k D_k
<\infty$.  Then by  dominated convergence, the desired
interchangeability follows.  Finally, by combining \eqref{e:Ljk-G} and
the bound on $\sup_{\sph{d-1}} |S_{j,k}|$, for any constant $M>0$,
\begin{align*}
  \sumoi k \frac{M^k}{k!} \sum^{k\dg}_{j=0}
  |L_{jk}(1)|\sup_{\sph{d-1}} |S_{j,k}|
  \le
  \sumoi k \frac{M^k}{k!} \sum^{k\dg}_{j=0} [\const^k\times
  \Gamma(k\alpha)(j+1)^{d/2-1}] < \infty.
\end{align*}
Then from the form of $L_{jk}(1)$ in \eqref{e:Ljk(1)}, it easily
follows that the series \eqref{e:mvs-pdf<1} is uniformly \ac in
$\{x{:}~|x|\ge r\}$ for any $r>0$.

\subsection{Proof of Proposition \ref{p:pdf-lt1}}
\ref{p:pdf-lt1a}
From \eqref{e:Sjk-bound}, Euler's reflection formula \eqref
{e:Euler}, and log-convexity of the Gamma function, each term in
\eqref{e:proj-g} with $k\ge j/\alpha$ has absolute value no greater
than
\begin{align*}
  &
    \frac{2^{\alpha k}}{k! |x|^{k\alpha+d}}
    \Gamma((j+k\alpha+d)/2) \Gamma(1+(k\alpha-j)/2)\times \const^k
    \times (j+1)^{d/2-1}
  \\
  &\le
    \frac{\const^k}{k! |x|^{k\alpha+d}}
    \Gamma(k\alpha+d/2) (k\alpha+1)^{d-2}.
\end{align*}
Since $\alpha\in (0,1)$, for any $c>0$, the sum of the
above terms over $k\ge j/\alpha$ is uniformly convergent on
$\{x{:}~|x|\ge c\}$, yielding the proof of \ref{p:pdf-lt1a}.

\ref{p:pdf-lt1b}
Fix $r>0$.  From \ref{p:pdf-lt1a}, $\phi_j(r\,\cdot)$ is a continuous
function on $\sph{d-1}$.  Then by $S_{j,k}\in \Cal H_{j,d}$,
$\phi_j(r\,\cdot)\in \Cal H_{j,d}$.  On the other hand, $g(r\,\cdot)$
is bounded so it is in $L^2(\sph{d-1})$.  Thus, to prove
\ref{p:pdf-lt1b}, it suffices to show $\proj_j g(r\,\cdot) =
\phi_j(r\,\cdot)$.

Write $h = \proj_j g(r\,\cdot)$ and $b=d/2-1$.  Since
$\inf_{\sph{d-1}} \Re(V)>0$ by Lemma \ref{l:sup-sp}, from the
integral representation of $\proj_j$ in \eqref{e:proj} and
Fubini's theorem, for each $u\in\sph{d-1}$,
\begin{align*}
  h(u)
  &=
    \proj_j
    \Sbr{(2\pi)^{-d} \rint e^{-\iunit r |z| \ip {u_z} \cdot}
    e^{-|z|^\alpha V(u_z)}\,\dd z}(u)
  \\
  &=
    (2\pi)^{-d} \rint (\proj_j e^{-\iunit r|z|\ip {u_z}\cdot})(u)
    e^{-|z|^\alpha V(u_z)}\,\dd z.
\end{align*}

Change the last integral into polar and coordinates and
plug in \eqref{e:FJ-proj}.  Then
\begin{align*}
  h(u)
  =
  \frac{(2\pi)^{-d/2}}{r^b}
  \intzi s^{d/2} (-\iunit)^j  J_{j+b}(r s)\Sbr{
  \frac{c_{j,d}}{A(\sph{d-1})}\sint
  \wt C^b_j(\ip u v) e^{-s^\alpha V(v)}\,\omega(\dd v)}\,\dd s.
\end{align*}
Then by \eqref{e:proj} again,
\begin{align} \label{e:proj-pdf}
  h(u)
  =
  \frac{(2\pi)^{-d/2} (-\iunit)^j}{r^b} \intzi s^{d/2} J_{j+b}(r
  s)\proj_j(e^{-s^\alpha V})(u)\,\dd s.
\end{align}
Note that by Fubini's theorem, each integral in the above three
displays is well-defined in the $L^1$ sense.  Then by dominated
convergence, $h(u) = \lim_{\rx\dto0} h_\rx(u)$, where
\begin{align*}
  h_\rx(u)
  &=
    \frac{(2\pi)^{-d/2} (-\iunit)^j}{r^b}
    \intzi s^{d/2} J_{j+b}(rs) \proj_j(e^{-s^\alpha
    V})(u) e^{-\rx s}\,\dd s
  \\
  &=
    \frac{(2\pi)^{-d/2} (-\iunit)^j}{r^b}
    \intzi s^{d/2} J_{j+b}(rs) \sumzi k
    \frac{(-s^\alpha)^k}{k!} (\proj_j V^k)(u) e^{-\rx s}\,\dd s
  \\
  &=
    \frac{(2\pi)^{-d/2}}{r^{k\alpha+d}}
    \sumzi k \frac{(-1)^k S_{j,k}(u)}{k!}
    \intzi s^{k\alpha+d/2} J_{j+b}(s) e^{-\delta s}\,\dd s
\end{align*}
with $\delta = \rx/r$.  From \eqref{e:I_r},
\begin{align*}
  h_\rx(u) =
  \sumzi k \frac{(-1)^k (2\pi)^{-d/2}}{r^{k\alpha+d} k!}
  \frac{\Grp{1+\delta^2}^{-(j+k\alpha+d)/2}}{2^{j-1+d/2}}
  L_{jk}\Grp{\nth{1+\delta^2}} S_{j,k}(u),
\end{align*}
and from the last paragraph in the proof of Theorem \ref{t:pdf-lt1},
each term on the \rhs has absolute value bounded by $C^k \sint
|V|^{2k}\, \dd\omega\cdot \Gamma(k\alpha)/(r^{k\alpha+ d} k!)$, where
$C$ is a constant.  Then by the same dominated convergence argument
following \eqref{e:I_r}, the proof of \ref{p:pdf-lt1b} follows.

\ref{p:pdf-lt1c}
The joint \pdf of $|X|$ and $u_X$ is $f(r,\theta) = r^{d-1}
g(r\theta)$, which is bounded on $\sph{d-1}$ for given $r>0$.
Integrate over $\theta\in\sph{d-1}$.  From \ref{p:pdf-lt1b} and $\sint
\phi\, \dd\omega=0$ for $\phi\in \Cal H_{j,d}$ if $j\ge 1$,
$g_{|X|}(r) = r^{d-1} \sint \phi_0(rv)\,\omega(\dd v)$.  By
\ref{p:pdf-lt1a}, the last integral can be  done term by term for the
series in \eqref{e:proj-g}.  By \eqref{e:proj}, $S_{0,k}$ is constant
$\nth{A(\sph{d-1})} \sint V^k\,\dd\omega$.  Then $\sint S_{0,k}\,
\dd\omega =  \sint V^k\,\dd\omega$ and \eqref{e:pdf-abs-lt1} follows.

\subsection{Proof of Proposition \ref{p:asym-lt1}}
Fix $n\in\Ints_+$.  By Taylor's theorem, for $t\in\Reals$,
\begin{align*}
  e^{-\iunit t}
  =
  \sum^n_{k=0} \frac{(-\iunit t)^k}{k!}
  + \nth{n!} \int^t_0 s^n (-\iunit)^{n+1} e^{-\iunit
  (t-s)}\,\dd s.
\end{align*}
Denote the integral on the \rhs by $R_n(t)$.  Then $|R_n(t)| =
O(|t|^{n+1})$.  As a result, for $x$, $z\in\Reals^d$, $|R_n(\ip x z)|
= |z|^{n+1} O(|x|^{n+1})$.  Since by Lemma  \ref{l:ReV},
$\inf_{\sph{d-1}} \Re(V) >0$, then $\rint |z|^{n+1} e^{-|z|^\alpha
  \Re(V(u_z))}\,\dd z<\infty$.  Then \label{e:g(x)}
\begin{align*}
  g(x)
  &=(2\pi)^{-d}     
    \rint
    \Sbr{\sum^n_{k=0} \frac{(-\iunit\ip x z)^k}{k!} + R_n(\ip x z)}
    e^{-|z|^\alpha V(u_z)}\,\dd z
  \\
  &=
    (2\pi)^{-d}\sum^n_{k=0} \frac{(-\iunit)^k}{k!}
    \rint \ip x z^k e^{-|z|^\alpha V(u_z)}\,\dd z
    + O(|x|^{n+1}).
\end{align*}
Given $k=0,\ldots,n$, in polar coordinates,
\begin{align*}
  \rint \ip x z^k e^{-|z|^\alpha V(u_z)}\,\dd z
  &=
    |x|^k\sint \ip {u_x} v^k \Grp{
    \intzi s^{k+d-1} e^{-V(v) s^\alpha}\,\dd s}\omega(\dd v)
  \\
  &=\frac{\Gamma(\frac{k+d}\alpha)}\alpha |x|^k
    \sint \ip {u_x} v^k [V(v)]^{-(k+d)/\alpha}\,\omega(\dd v).
\end{align*}
Since $\ip{u_x} \cdot$ and $V^{-(k+d)/\alpha}$ are both in
$L^2(\sph{d-1})$, by \eqref{e:V-SH}
\begin{align}\label{e:power-exp}
  \rint \ip x z^k e^{-|z|^\alpha V(u_z)}\,\dd z
  =\frac{\Gamma(\frac{k+d}\alpha)}\alpha |x|^k
  \sumzi j \iunit^j
  \sint \ip {u_x} v^k S_{j, -(k+d)/\alpha}(v)
  \,\omega(\dd v).   
\end{align}
Put $P_j = S_{j, -(k+d)/\alpha}$.  Because $P_j$ is a homogeneous
polynomial with degree $j$ and given $u\in\sph{d-1}$, $\ip u \cdot^k$
is a homogeneous polynomial with degree $k$, if $k-j$ is odd, then the
integral $\sint \ip u v^k P_j(v) \,\omega(\dd v)$ is equal to 0, 
while if $k-j$ is even, it is equal to
\begin{align*}
  \sint |\ip u v|^k (\sgn\ip u v)^{\rx_j}
  P_j(v)\,\omega(\dd v).
\end{align*}
From \eqref{e:Funk-Hecke} and the first expression in
\eqref{e:Funk-Hecke2}, the integral is equal to $w_j(k) P_j(u)$
with
\begin{align*}
  w_j(k)
  =
  \frac{\pi^{d/2}\Gamma(k+1)}
    {2^{k-1}\Gamma(\frac{k+j+d}2)\Gamma(\frac{k-j}2+1)}.
\end{align*}
Write $j=k-2m$, where $m$ is an integer.  Then $\iunit^j = \iunit^k
(-1)^m$ and $w_{k-2m}(k)=0$ if $m<0$.  Meanwhile, since $j\ge0$,
then $m\le k/2$.  Combine all the information with the \rhs of the
above display for $g(x)$.  Then \eqref{e:asym-lt1} follows.  Finally,
let $n=0$ in \eqref{e:asym-lt1}.  Notice that for $u\in\sph{d-1}$,
$S_{0, -d/\alpha}(u)$ is the constant $\nth{A(\sph{d-1})} \int
V^{-d/\alpha} \,\dd\omega$.  Then by the continuity of $g$ at $0$,
the expression of $g(0)$ is obtained.

\section{Density when $\alpha\in (1,2)$}  \label{s:pdf(1,2)}
In this section, let $\mu$ be a nondegenerate $\alpha$-stable
distribution with $\alpha\in (1,2)$.  As before, let $\mu$ have no
shift, so that $\Phi_\mu(z) = |z|^\alpha V(u_z)$ as in
\eqref{e:che-ss3}.

With $\alpha>1$, the first step to get the \pdf of $\mu$ follows the
one for the univariate case; cf.\ \cite {uchaikin:99:vsp}, Section~4.2
and also the proof of Proposition \ref{p:asym-lt1}.  That is, in
evaluating the inverse transform \eqref{e:ift-ss}, expand
$e^{-\iunit\ip x z}$ into a power series of $z$ and then integrate
term by term.  This leads to the result below.   Let the spherical
harmonics $S_{j,p}$ be as in  \eqref{e:Sjp}.

\begin{prop} \label{p:pdf-gt1}
  Let $\mu$ be a nondegenerate $\alpha$-stable distribution on
  $\Reals^d$, $d\ge2$, with $\alpha\in(1,2)$ and characteristic
  exponent \eqref{e:che-ss3}.  Then the \pdf of $\mu$ is 
  \begin{align}\label{e:pdf>1b}
    g(x)=
    \frac{2^{1-d}}{\pi^{d/2} \alpha}
    \sumzi n \Gamma((n+d)/\alpha) \Grp{\frac{|x|}{2}}^n
    \sum^{\Flr{n/2}}_{m=0}
    \frac{(-1)^m S_{n-2m,-(n+d)/\alpha}(u_x)}{m!\Gamma(n-m+d/2)},
    \quad x\in\Reals^d.
  \end{align}
  The series is uniformly \ac in $\{x{:}~|x|\le M\}$ for any $M>0$.
  If $X\sim \mu$, then $|X|$ has \pdf
  \begin{align} \label{e:pdf-abs-gt1}
    g_{|X|}(r) = \nth{\pi^{d/2}\alpha}
    \sumzi n \frac{(-1)^n \Gamma((2n+d)/\alpha)}{n! \Gamma(n+d/2)}
    \Grp{\frac{r}{2}}^{2n+d-1}
    \sint \frac{\dd\omega}{V^{(2n+d)/\alpha}}.
  \end{align}
\end{prop}

For spherically symmetric $\alpha$-stable distributions, the following
construction based on subordination is well known
(cf.~\cite{sato:99:cup}, section 30).  Let $Z = (\eno Z d)$, where
$Z_i\sim N(0,1)$ are independent.  Let $\zeta$ be a positive
$(\alpha/2)$-stable random variable independent of $Z$ such that for
$t>0$, $\mean(e^{-t\zeta}) = e^{-t^{\alpha/2}}$.  Given constant
$V_0>0$, for $z\in \Reals^d$, $\mean(e^{\iunit \ip z {V^{1/\alpha}_0
    \sqrt{2\zeta} Z}})= \mean(e^{-V^{2/\alpha}_0 |z|^2\zeta}) =
e^{-|z|^\alpha V_0}$.  As a result, $V^{1/\alpha}_0 \sqrt{2\zeta} Z$
is spherically symmetric and $\alpha$-stable with characteristic
exponent $\Phi(z) = |z|^\alpha V_0$.  For more general $\alpha$-stable
distributions, Proposition \ref{p:pdf-gt1} has the useful consequence
below.
\begin{cor} \label{c:pdf-gt1}
  The \pdf in \eqref{e:pdf-abs-gt1} can be written as
  \begin{align} \label{e:pdf-abs-gt1b}
    g_{|X|}(r) = \frac{(r/2)^{d-1}}{2\pi^{d/2}}
    \sint \nth{V^{d/\alpha}}\mean\Sbr{\zeta^{-d/2}
    \exp\Cbr{-\frac{r^2}{4 V^{2/\alpha}\zeta}}}\,\dd\omega.
  \end{align}
\end{cor}

From \eqref{e:proj}, the series in \eqref{e:pdf>1b} consists of
integrals of the form $\int \ip u v^k [V(u)]^{-p} \,\omega(\dd u)$
with $k\in\Ints_+$ and $p>0$.  No closed form expressions of the
integrals are known except when $V$ is a constant $c$, in which case
$S_{n-2m, -(n+d)/\alpha} = c^{-(n+d)/\alpha}$ if $n=2m$ and 0
otherwise, yielding the \pdf of a spherically symmetric stable
distribution (\cite {uchaikin:99:vsp}, p.~212).  To get a more
explicit representation, one way is to expand $h(z) = z^{-p}$ into a
power series, plug in $z = V$, and integrate the resulting series term
by term.  Because $h(z)$ is singular at 0 and $V$ only takes values in
the right half of the complex plane, one may expand $h(z)$ at some
$R\in (0,\infty)$.  Since the resulting power series is convergent
only in the disc $\{z{:}~|z-R|<R\}$, it is necessary that $|V-R|<R$.
This is guaranteed by the result below.

\begin{lemma} \label{e:V-R}
  Let $\varrho(z) = \lfrac{|z|^2}{(2\Re z)}$.  Then
  \begin{align*}
    0<\varrho_0 = \sup_{\sph{d-1}} \varrho(V)
    <
    \frac{\lambda(\sph{d-1})}{1+\cos(\pi\alpha)} < \infty.
  \end{align*}
  Moreover, given any $R>\varrho_0$, letting $V_* = R - V$,
  $\sup_{\sph{d-1}} |V_*| < R$.
\end{lemma}

Fixing $R$ as in Lemma \ref{e:V-R}, $V^{-p} = (R - V_*)^{-p}$ can be
expanded into a power series of $V_*$ that can be integrated term by
term.  Since $[V_*(-u)]^k = \conj{[V_*(u)]^k}$, similar to 
\eqref{e:V-SH},
\begin{align}\label{e:V-SH*}
  \begin{split}
    S^*_{j,k}
    &= (-\iunit)^j \proj_j (V^k_*)\in \Cal H_{j,d}\
    \text{is real-valued}, \quad j\ge0,\\
    V^k_*
    &= \sumzi j \iunit^j S^*_{j,k} \text{~in $L^2(\sph{d-1})$}.
  \end{split}
\end{align}

\begin{theorem} \label{t:pdf-gt1}
  Let $\mu$ be a nondegenerate $\alpha$-stable distribution on
  $\Reals^d$, $d\ge2$, with $\alpha\in(1,2)$ and characteristic
  exponent \eqref{e:che-ss3}.  Let $R$ and $S^*_{j,k}$ be defined as above.
  Then the \pdf of $\mu$ can be written as
  \begin{multline}\label{e:pdf>1}
    g(x)
    =
    \frac{2^{1-d}}{\pi^{d/2}\alpha}
    \sumzi n \Grp{\frac{|x|}{2}}^n
    \sum^{\Flr{n/2}}_{m=0} \frac{(-1)^m}{m!\Gamma(n-m+d/2)}
    \sumzi k \frac{\Gamma(k+(n+d)/\alpha)}{k! R^{k+(n+d)/\alpha}}
    S^*_{n-2m,k}(u_x),\\
    x\in\Reals^d.
  \end{multline}
  The series is uniformly \ac in $\{x{:}~|x|\le M\}$ for any $M>0$.
  If $X\sim\mu$, then $|X|$ has \pdf
  \begin{align}\label{e:pdf-abs>1}
    g_{|X|}(r)
    =
    \frac{(r/2)^{d-1}}{\pi^{d/2}\alpha}
    \sumzi n 
    \frac{(-r^2/4)^n}{n!\Gamma(n+d/2)}
    \sumzi k \frac{\Gamma(k+(2n+d)/\alpha)}{k! R^{k+(2n+d)/\alpha}}
    \sint V^k_*\,\dd\omega.
  \end{align}
\end{theorem}

\begin{proof}[Remark]
  If $\mu$ has a polynomial spectral spherical density $P$ of degree
  $\dg$, i.e., $P\in\Cal P_{\dg, d}$, then by \eqref{e:V-sphere},
  $V_* = R - V\in\Cal P_{\dg,d}$.  Consequently, the first remark that
  follows Theorem \ref{t:pdf-lt1} applies to $S^*_{j,k}$.  Similarly,
  $\sint V^k_*\,\dd\omega$ can be expressed in closed form.
\end{proof}

Finally, consider the asymptotic expansion of $g(x)$ as $|x|\toi$.  It
is well known that all the partial derivatives of $g(x)$ tend to 0 as 
$|x|\toi$ (\cite{sato:99:cup}, Proposition 28.1).  However, it is hard
to get a grasp on the asymptotic behavior of $g(x)$ from
Proposition \ref {p:pdf-gt1} or Theorem \ref{t:pdf-gt1}, as their
representations of $g(x)$ are in terms of positive integral powers of
$|x|$.  Since given $|x|$, $g$ is characterized by its behavior on the
sphere $|x|\cdot \sph{d-1}$, one essentially needs to consider the
asymptotic expansion of the function $g(r\,\cdot)$ as $r\toi$.  The
following partial result supplies an asymptotic
expansion of the spherical harmonic of $g(r\,\cdot)$ of a given
degree. 
\begin{prop} \label{p:asym-gt1}
  Let $\mu$ be a nondegenerate $\alpha$-stable distribution on
  $\Reals^d$, $d\ge2$, with $\alpha\in(1,2)$, characteristic exponent
  \eqref{e:che-ss3}, and \pdf $g$.  Fix $j\in\Ints_+$.  Then for each
  $n\ge0$, as $r\toi$,
  \begin{align} \label{e:asym-proj-gt1}
    \sup_{\sph{d-1}} \Abs{
    \proj_j g(r\,\cdot) -
    \sum^n_{k=0} \frac{(-2^\alpha)^k \pi^{-d/2}}{k! r^{k\alpha+d}}
    \frac{\Gamma((j+k\alpha+d)/2)}{\Gamma((j-k\alpha)/2)}
    S_{j,k}}= O(r^{-(n+1)\alpha-d}).
  \end{align}
  Furthermore, if $g_{|X|}(r)$ denotes the \pdf of $|X|$, then as
  $r\toi$,
  \begin{align} \label{e:asym-radius-gt1}
    g_{|X|}(r) = 
    \sum^n_{k=0} \frac{(-2^\alpha)^k \pi^{-d/2}}{k! r^{k\alpha+d}}
    \frac{\Gamma((1+k\alpha+d)/2)}{\Gamma((1-k\alpha)/2)}
    \int V^k\,\dd\omega + O(r^{-(n+1)\alpha-1}).
  \end{align}
\end{prop}

It would be more satisfactory to have a full asymptotic expansion of
$g$.  Eq.~\eqref{e:asym-proj-gt1} indicates that the expansion would
have the form as the series representation \eqref{e:mvs-pdf<1},
and in view of \cite {rvaceva:62:stmsp}, Theorem 4.2 or \eqref
{e:pdf-asym-lt1}, its first term would be $2\pi^{-1}
\alpha\Gamma(\alpha) \sin(\pi\alpha/2) r^{-\alpha-d} P$, where $P$ is
the spectral spherical density of $\mu$.  Thus, if $\mu$ has no
spectral spherical density, an asymptotic expansion of $g$ is unlikely
to exist.  On the other hand, even if $\mu$ has a polynomial spectral
spherical density $P$, except for $P$ being a constant, the method of
the paper has yet been able to yield the desired expansion.

\subsection{Proof of Proposition \ref{p:pdf-gt1}} \label{ss:p:pdf-gt1}
Given $x$, since $\rint e^{|\ip x z| - |z|^\alpha \Re(V(u_z))}\,\dd
z<\infty$ due to $\alpha>1$ and $\inf_{\sph{d-1}}\Re(V)>0$, by
dominated convergence followed by changing to polar coordinates,
\begin{align*}
  g(x)
  =(2\pi)^{-d}\sumzi n \frac{(-\iunit)^n}{n!}
  \rint \ip x z^n e^{-|z|^\alpha V(u_z)}\,\dd z.
\end{align*}
Treat each of the integrals on the \rhs the same way as in the proof
of Proposition \ref {p:asym-lt1}; see the argument from
\eqref{e:power-exp} to the end of that proof.  Then \eqref{e:pdf>1b}
follows.  Given $M>0$, to show the uniform and absolute convergence of
\eqref{e:pdf>1b} on $\{x{:}~|x|\le M\}$, for $0\le m\le \Flr{n/2}$,
similar to \eqref{e:Sjk-bound},
\begin{align*}
  \sup_{\sph{d-1}} |S_{n-2m,-(n+d)/\alpha}|
  \le
  \sqrt{c_{n-2m,d}}
  \sup_{\sph{d-1}} |V^{-(n+d)/\alpha}|
  \le
  \frac{\const\times (n-2m+1)^{d/2-1}}
  {[\inf_{\sph{d-1}}\Re(V)]^{(n+d)/\alpha}}.
\end{align*}
Then by Lemma \ref{l:ReV}, there is a constant $C>0$, such
that $\sup_{\sph{d-1}} |S_{n-2m,-(n+d)/\alpha}|\le C^n$ for all $n$
and $0\le m\le n/2$.  Consequently, for all $|x|\le M$ and $0\le m\le
n/2$,
\begin{align*}
  \Abs{
  \Gamma((n+d)/\alpha) \Grp{\frac{|x|}{2}}^n
  \frac{(-1)^m S_{n-2m,-(n+d)/\alpha}(u_x)}{m!\Gamma(n-m+d/2)}}
  \le
  \frac{\Gamma((n+d)/\alpha) (CM/2)^n}{m!\Gamma(n-m+d/2)},
\end{align*}
so the uniform and absolute convergence of the series \eqref{e:pdf>1b}
follows from
\begin{align} \label{e:pdf-gt1b-finite}
  \sum_{0\le m\le n/2} 
  \frac{\Gamma((n+d)/\alpha) C^n}{m!\Gamma(n-m+d/2)}
  <\infty \quad\text{for any~$C>0$}.
\end{align}
To see \eqref{e:pdf-gt1b-finite}, first, by log-convexity of the
Gamma function,
\begin{align*}
  m!\Gamma(n-m+d/2) \ge [\Gamma((n+1+d/2)/2)]^2 \quad\text{for~}
  0\le m\le n/2.
\end{align*}
Since $\Gamma(x)$ is increasing on $[1,\infty)$ (\cite{NIST:10},
\S5.4(iii)) and $d\ge2$, by Legendre's duplication formula
\eqref{e:L-dup}, for $n\ge1$,
$[\Gamma(n+1+d/2)/2]^2\ge \Gamma((n+2)/2)\Gamma((n+1)/2) =
2^{-n}n!\sqrt\pi$.  Then \eqref{e:pdf-gt1b-finite} follows by
$\alpha>1$.

From \eqref{e:pdf>1b}, the joint \pdf of $|X|$ and $u_X$ is
\begin{align*}
  r^{d-1} g(r\theta)
  =
  \frac{(r/2)^{d-1}}{\pi^{d/2} \alpha}
  \sum_{0\le m\le n/2} \Gamma((n+d)/\alpha) (r/2)^n
  \frac{(-1)^m S_{n-2m,-(n+d)/\alpha}(\theta)}{m!\Gamma(n-m+d/2)}. 
\end{align*}
Given $r>0$, since the series in \eqref{e:pdf>1b} is uniformly \ac on
$\{x{:}~|x|=r\}$, its integral over $\theta$ can be done term-by-term,
so that
\begin{align*}
  g_{|X|}(r)
  &=\sint r^{d-1} g(r\theta)\,\omega(\dd\theta)
    \\
  &=
    \frac{(r/2)^{d-1}}{\pi^{d/2} \alpha}
    \sum_{0\le m\le n/2} 
    \Gamma((n+d)/\alpha) (r/2)^n
    \sint\frac{(-1)^m S_{n-2m,-(n+d)/\alpha}(\theta)}
    {m!\Gamma(n-m+d/2)}\, \omega(\dd\theta). 
\end{align*}
Since $\sint S_{0, -(n+d)/\alpha}\,\dd\omega = \sint
V^{-(n+d)/\alpha}\, \dd\omega$ and $\sint S_{n-2m,-(n+d)/\alpha}
\,\dd\omega=0$ if $n>2m$, the \pdf of $|X|$ in \eqref{e:pdf-abs-gt1}
is obtained.

\subsection{Proof of Corollary \ref{c:pdf-gt1}}
Since $\zeta$ has characteristic exponent $|t|^{\alpha/2}
\exp\{-\iunit (\alpha/2)(\pi/2)\sgn t\}$, the Mellin transform of its 
\pdf, denoted $q$, yields that for all $s\in\Coms$ with $-1 < \Re s <\alpha/2$,
\begin{align*}
  \mean(\zeta^s) = \frac{2\Gamma(-2s/\alpha)}{\alpha\Gamma(-s)}
\end{align*}
(\cite{uchaikin:99:vsp}, section 5.6).  From \cite{uchaikin:99:vsp},
Theorem 5.4.1, $q(x) = O(x^{-(4-\alpha)/(4-2\alpha)} e^{-c
  x^{-\alpha/(2-\alpha)}})$, $x\to0+$, where $c$ is a constant.  Then
for all $s$ with $\Re s<\alpha/2$,
$\zeta^s$ is integrable and so the above identity holds.  Then, as the
summation and integration in \eqref{e:pdf-abs-gt1} are
interchangeable,
\begin{align*}
  g_{|X|}(r)
  &= \nth{2\pi^{d/2}}\sint
    \Sbr{\sumzi n \frac{(-1)^n}{n!}
    \frac{2\Gamma((2n+d)/\alpha)}{\alpha\Gamma(n+d/2)}
    \Grp{\frac{r}{2}}^{2n+d-1}
    \frac{1}{V^{(2n+d)/\alpha}}}\dd\omega
    \\
  &= \nth{2\pi^{d/2}}\sint
    \Sbr{\sumzi n \frac{(-1)^n}{n!}
    \mean(\zeta^{-n-d/2})
    \Grp{\frac{r}{2}}^{2n+d-1}
    \frac{1}{V^{(2n+d)/\alpha}}}\dd\omega,
\end{align*}
yielding \eqref{e:pdf-abs-gt1b}.

\subsection{Proof of Lemma \ref{e:V-R}}
By Lemma \ref{l:ReV}, $\inf_{\sph{d-1}} \Re(V)>0$, giving
$\varrho_0>0$.  It is easy to check that if $\Re z>0$, then $|z -
\varrho(z)| = \varrho(z)$.  Geometrically, $\varrho(z)$ is the radius
of the unique circle that is centered on $\Reals$ and passes both $0$
and $z$.  Then for any $r>\varrho(z)$, the circle with center and
radius both equal to $r$ encircles $z$, i.e., $|z - r|<r$.  Provided
that $\varrho_0<\infty$, from the geometric interpretation, the last
claim of the lemma holds.  Thus, it only remains to show the upper
bound on $\varrho_0$.  From \eqref{e:che-ss} and $\alpha\in (1,2)$,
\begin{gather*}
  |V(u)|<
  \nth{|\cos(\pi\alpha/2)|}
  \sint |\ip u v|^\alpha \,\lambda(\dd v)
  <
  \nth{|\cos(\pi\alpha/2)|}
  \sint |\ip u v|\,\lambda(\dd v),
  \\
  \Re(V(u)) = 
  \sint |\ip u v|^\alpha \,\lambda(\dd v)
  >
  \sint \ip u v^2\,\lambda(\dd v),
\end{gather*}
where all the inequalities are strict.  By Cauchy--Schwarz
inequality, the first line yields
\begin{align*}
  |V(u)|^2 <
  \frac{\lambda(\sph{d-1})}{|\cos(\pi\alpha/2)|^2} \sint
  \ip u v^2\, \lambda(\dd v).
\end{align*}
Then for each $u\in\sph{d-1}$,
\begin{align*}
  \varrho(V(u)) = \frac{|V(u)|^2}{2\Re(V(u))}
  < \frac{\lambda(\sph{d-1})}{2|\cos(\pi\alpha/2)|^2}
  =\frac{\lambda(\sph{d-1})}{1+\cos(\pi\alpha)}.
\end{align*}
By the continuity of $\varrho(V)$ on $\sph{d-1}$, the upper bound on
$\varrho_0$ follows.

\subsection{Proof of Theorem \ref{t:pdf-gt1}}
By the definitions of $S_{n-2m, -(n+d)/\alpha}$ and $V_*$,
\begin{align*}
  &\Gamma((n+d)/\alpha) S_{n-2m, -(n+d)/\alpha}
  =(-\iunit)^{n-2m} \Gamma((n+d)/\alpha)
    \proj_{n-2m}(V^{-(n+d)/\alpha})
  \\
  &\hspace{3cm}=
    (-\iunit)^{n-2m} R^{-(n+d)/\alpha} \Gamma((n+d)/\alpha)
    \proj_{n-2m}((1-V_*/R)^{-(n+d)/\alpha}).
\end{align*}
Then by $S^*_{n-2m,k} = (-\iunit)^{n-2m} \proj_{n-2m}(V^k_*)$ and
$\sumzi k \frac{(c)_k}{k!} (-z)^k = (1+z)^{-c}$ for $c>0$ and $|z|<1$,
the \rhs equals
\label{e:gt1S}
\begin{align*}
  &(-\iunit)^{n-2m} R^{-(n+d)/\alpha}\, \Gamma((n+d)/\alpha)\,
    \proj_{n-2m} \Sbr{
    \sumzi k \frac{((n+d)/\alpha))_k}{k!}
    (V_*/R)^k
    }
  \\
  &=
  \sumzi k \frac{\Gamma(k+(n+d)/\alpha)}{k! R^{k+(n+d)/\alpha}}
    S^*_{n-2m,k}.
\end{align*}
This combined with \eqref{e:pdf>1b} then yields \eqref{e:pdf>1}.
Put $D=\sup_{\sph{d-1}}|V_*|$.  Similar to \eqref{e:Sjk-bound},
for any $n\ge0$ and $k\ge0$,  \label{e:S*}
\begin{align*}
  \sup_{\sph{d-1}} |S^*_{n-2m,k}|
  \le
  \sqrt{c_{n-2m,d}}\sup_{\sph{d-1}} |V_*|^k
  \le \const\times (n+1)^{d/2} D^k.
\end{align*}
Since $D<R$, then for any $M>0$,
\begin{align*}
  &
    \sumzi n M^n
    \sum^{\Flr{n/2}}_{m=0} \frac{1}{m!\Gamma(n-m+d/2)}
    \sumzi k \frac{\Gamma(k+(n+d)/\alpha)}{k! R^{k+(n+d)/\alpha}}
    \sup_{\sph{d-1}} |S^*_{n-2m,k}|
  \\
  &\le
    \const\times\sumzi n (n+1)^{d/2} M^n \sum^{\Flr{n/2}}_{m=0}
    \frac{\Gamma((n+d)/\alpha)}{m!\Gamma(n-m+d/2)
    R^{(n+d)/\alpha}}
    \sumzi k \frac{((n+d)/\alpha)_k}{k!} \Grp{\frac{D}{R}}^k
    \\
  &=
    \const\times\sumzi n (n+1)^{d/2} M^n \sum^{\Flr{n/2}}_{m=0}
    \frac{\Gamma((n+d)/\alpha)}{m!\Gamma(n-m+d/2)}
    (R-D)^{-(n+d)/\alpha},
\end{align*}
which is finite by similar argument for \eqref{e:pdf-gt1b-finite}.  It
follows that the series in \eqref{e:pdf>1} is uniformly \ac in
$\{x{:}~|x|\le M\}$.  Finally, \eqref{e:pdf-abs>1} follows by similar
argument for \eqref{e:pdf-abs-gt1}.

\subsection{Proof of Proposition
  \ref{p:asym-gt1}}  \label{ss:asym-gt1}
For $r>0$ and $u\in\sph{d-1}$, by change of variable,
\begin{align} \label{e:g-G}
  g(ru)
  =
  (2\pi)^{-d} r^{-d}\rint e^{-\iunit \ip u z}
  e^{-r^{-\alpha} |z|^\alpha V(u_z)}\,\dd z
  =
  (2\pi)^{-d/2} r^{-d} G_{r^{-\alpha}}(u),
\end{align}
where for $t>0$,
\begin{align*} 
  G_t(u)
  =
  (2\pi)^{-d/2}
  \rint e^{-\iunit \ip u z} e^{-t |z|^\alpha V(u_z)}\,\dd z.
\end{align*}
From \eqref{e:proj-pdf}, for $j\in\Ints_+$,
\begin{align} \label{e:proj-Gt}
  (\proj_j G_t)(u)
  =(-\iunit)^j \intzi s^{d/2} J_{j+b}(s)
  \proj_j(e^{-ts^\alpha V})(u)\,\dd s.
\end{align}

\begin{lemma} \label{l:Gtlimit}
  Given $k\in\Ints_+$, as $t\to0+$,
  \begin{align*}
    \frac{\partial^k}{\partial t^k} (\proj_j G_t)(u)
    \to
    \frac{2^{d/2}(-2^\alpha)^k  \Gamma((j+k\alpha+d)/2)}
    {\Gamma((j-k\alpha)/2)} S_{j,k}(u)
  \end{align*}
  uniformly in $u\in\sph{d-1}$.
\end{lemma}

\begin{lemma} \label{l:derivative}
  Suppose $f: (0,1]\to \Reals$ is $n$ times differentiable such that
  for each $0\le k\le n$, $f\Sp k(t)$ converges as $t\to0+$.  Let the
  limits be $a_k$.  Then by defining $f(0)=a_0$, $f$ is $n$ times
  differentiable at $0$ with (one-sided) derivatives $f\Sp k(0) =
  a_k$, $0\le k\le n$.
\end{lemma}
\begin{proof}
  For each $t\in (0,1)$, by the mean value theorem, there is $s\in
  (0,t)$ such that $f'(s) = [f(t)-f(0)]/t$.  As $t\to0+$, $s\to0+$ and
  so by assumption, $[f(t)-f(0)]/t\to a_1$.  This shows $f'(0)=a_1$.
  The proof for higher order derivatives of $f$ at 0 follows by
  induction.
\end{proof}

Assuming Lemma \ref{l:Gtlimit} is true for now, by Lemma
\ref{l:derivative} and Taylor's theorem, given $n\ge0$, for
$u\in\sph{d-1}$ and $t>0$,
\begin{align*}
  (\proj_j G_t)(u)
  =
  \sum^n_{k=0} 
  \frac{2^{d/2}(-2^\alpha)^k  \Gamma((j+k\alpha+d)/2)}
  {\Gamma((j-k\alpha)/2)} S_{j,k}(u) \frac{t^k}{k!}
  +
  \frac{\partial^{k+1}}{\partial t^{k+1}} (\proj_j
  G_{t^*})(u) \frac{t^{k+1}}{(k+1)!},  
\end{align*}
where $t^*=t^*(u)$ is some number in $(0,t)$.  As $t\to0+$,
$t^*\to0+$, 
so by Lemma \ref{l:Gtlimit}, the remainder term is $O(t^{k+1})$ with
the implicit constant being uniform in $u$.  From the relationship
\eqref{e:g-G}, \eqref{e:asym-proj-gt1} follows.  Since by
\eqref{e:proj},
\begin{align*}
  g_{|X|}(r) = r^{d-1} \sint g(r u)\,\omega(\dd u)
  = r^{d-1} 
  A(\sph{d-1})\proj_0 g(r\,\cdot),
\end{align*}
by \eqref{e:asym-proj-gt1},
\begin{align*}
  g_{|X|}(r) = 
  A(\sph{d-1})
  \sum^n_{k=0} \frac{(-2^\alpha)^k \pi^{-d/2}}{k! r^{k\alpha+d}}
  \frac{\Gamma((1+k\alpha+d)/2)}{\Gamma((1-k\alpha)/2)} S_{0,k}
  + O(r^{-(n+1)\alpha-1}).
\end{align*}
Since $S_{0,k} = \nth{A(\sph{d-1})}\sint V^k\,\dd\omega$,
\eqref{e:asym-radius-gt1} follows, completing the proof of Proposition
\ref{p:asym-gt1}.

To prove Lemma \ref{l:Gtlimit}, the following result is needed.
\begin{lemma} \label{l:Bessel-limit}
  Let $\alpha\in (1,2)$.  Given $\tau>0$, $p>0$, for $\xi\in\Coms$
  with $\Re\xi>0$ and $s>0$, let
  \begin{align*}
    f_s(\xi) = \intzi J_\tau(r) r^{p-1} e^{-s r^\alpha\xi}\,\dd r.
  \end{align*}
  Then, as $s\to0+$,
  \begin{align} \label{e:Bessel-limit}
    f_s(\xi)\to
    \frac{2^{p-1} \Gamma((\tau+p)/2)}{\Gamma(1 + (\tau-p)/2)}
  \end{align}
  uniformly in any compact set of $\xi\in\{z\in\Coms{:}\ \Re z>0\}$.
\end{lemma}

For $\alpha=1$, the limit is essential for the proof of Theorem
\ref{t:pdf-lt1}; see the passage below \eqref{e:I_r}.  For $\alpha=2$,
the limit is also known (cf.\ \cite{andrews:99:cup}, Theorem 4.11.7
followed by Theorem 4.2.2).

\begin{proof}
  Let $r_0 =  \tau\pi/2+\pi/4$.  By \cite {andrews:99:cup},
  Eq.~(4.8.5), given $N\in\Nats$, $J_\tau(r) =  \Re(k_N(r)) + R_N(r)$,
  $r>0$, where
  \begin{align*}
    k_N(r)
    =
    \sqrt{\frac2{\pi r}}  \sum^N_{n=0} e^{\iunit(r - r_0)}
    \frac{(1/2-\tau)_n (1/2+\tau)_n}{n! (2 r\iunit)^n}
  \end{align*}
  and $R_N(r) = O(r^{-N-1})$ as $r\toi$.  Then
  \begin{align*}
    f_s(\xi)
    =
    \int^1_0 J_\tau(r) r^{p-1} e^{-s r^\alpha \xi}\,\dd r
    + \int^\infty_1 R_N(r) r^{p-1} e^{-s r^\alpha \xi}\,\dd r
    + \int^\infty_1 \Re(k_N(r)) r^{p-1} e^{-s r^\alpha \xi}\,\dd r.
  \end{align*}
  Since $\Re\xi>0$, $|e^{-s r^\alpha \xi}| \le1$.  By
  \eqref{e:Bessel3}, $|J_\tau(r)| r^{p-1} = O(r^{\tau+p-1})$ as
  $r\to0+$.  \label{e:B-asymp0} Also, clearly $|R_N(r)
  r^{p-1}| = O(r^{p-2-N})$ as $r\toi$.  Fix $N>p-1$.  Then by
  dominated convergence, as $s\to0+$, the sum of the first two
  integrals on the \rhs tends to $\int^1_0 J_\tau(r) r^{p-1}\,\dd r +
  \int^\infty_1 R_N(r) r^{p-1}\,\dd r$.  It is easy to see that the
  convergence is uniform in any compact set of $\xi\in \{z{:}\
  \Re z>0\}$.  For the third integral, since it is equal to
  \begin{align*}
    \nth{\sqrt{2\pi}}
    \sum^N_{n=0} \frac{(1/2-\tau)_n (1/2+\tau)_n}{n! 2^n}
    \int^\infty_1 [(-\iunit)^n e^{\iunit(r-r_0)}
    + \iunit^n e^{-\iunit(r-r_0)}]
    r^{p-n-3/2} e^{-s r^\alpha\xi}\, \dd r,
  \end{align*}
  it suffices to consider, for any given $q\in\Reals$, the convergence
  of the integrals
  \begin{align*}
    \int^\infty_1 e^{\pm\iunit r} r^q e^{-s r^\alpha \xi}\,\dd r.
  \end{align*}

  Put $H_s(z)=e^{\iunit z} z^q e^{-s z^\alpha \xi}$.  Then $H_s(z)$ is
  analytic on the right half of $\Coms$.  Given a compact
  $D\subset\{z{:}\ \Re z>0\}$, fix $\theta_*\in (0,
  \pi/2)$ such that $\alpha\theta_* + |\arg \xi| < \pi/2$ for all
  $\xi\in D$.  Given $R>1$, consider the contour consisting of
  $[1,R]$, $\{r e^{\iunit\theta_*}{:}\ r\in [1, R]\}$, $C =
  \{e^{\iunit\theta}{:}\ \theta \in [0, \theta_*]\}$, and $C_R = \{R
  e^{\iunit \theta}{:}\ \theta\in [0,\theta_*]\}$.  Fixing $s>0$, for
  $z\in C_R$, $|H_s(z)|= R^q e^{- R\sin(\arg z) - s R^\alpha
    |\xi|\cos(\alpha\arg z + \arg \xi)}$.  By $|\alpha\arg z + \arg
  \xi| \le \alpha\theta_* + |\arg \xi| < \pi/2$, $\inf_{z\in C_R}
  \cos(\alpha \arg z + \arg \xi)>0$.  Then by $\alpha>1$, as $R\toi$,
  $\int_{C_R} H_s(z)\,\dd z\to0$, and the integral of $H_s(z)$ along
  the ray from 1 to $R$ and the integral along the ray from
  $e^{\iunit\theta_*}$ to $Re^{\iunit\theta_*}$ both converge.
  Therefore,
  \begin{align*}
    \int^\infty_1 H_s(r)\,\dd r
    &=
      \Grp{\int_C + \int^{\infty\cdot e^{\iunit\theta_*}}_{1\cdot
      e^{\iunit\theta_*}}} H_s(z)\,\dd z
    \\
    &=
      \iunit \int^{\theta_*}_0 e^{\iunit e^{\iunit\theta} +
      \iunit(q+1)\theta - s \xi e^{\iunit\alpha\theta}}\,\dd\theta
      +
      e^{\iunit (q+1)\theta_*} \int^\infty_1 e^{\iunit r
      e^{\iunit\theta_*}} r^q e^{-s r^\alpha e^{\iunit \alpha\theta_*}
      \xi}\,\dd r,
  \end{align*}
  where the integral along $C$ is from 1 to $e^{\iunit \theta_*}$.  As
  $s\to0+$, the first integral on the \rhs converges uniformly in
  $\xi\in D$.  On the other hand,
  \begin{align*}
    |e^{\iunit r e^{\iunit\theta_*}}
    r^q e^{-s r^\alpha e^{\iunit\alpha
    \theta_*}\xi}|=e^{-r\sin\theta_*} r^q e^{-s r^\alpha
    |\xi|\cos(\alpha \theta_* +\arg \xi)}.
  \end{align*}
  By $\sin\theta_*>0$ and dominated convergence, the second integral
  on the \rhs converges uniformly in $\xi\in D$.  As a result, 
  \begin{align*}
    \int^\infty_1 H_s(r)\,\dd r
    \to 
    \iunit \int^{\theta_*}_0 e^{\iunit e^{\iunit\theta} +
    \iunit(q+1)\theta}\,\dd\theta  +
    e^{\iunit (q+1)\theta_*} \int^\infty_1 e^{\iunit r
    e^{\iunit\theta_*}} r^q\,\dd r
  \end{align*}
  uniformly in $\xi\in D$.  Likewise, by considering the contour
  consisting of $[1,R]$, $\{r e^{-\iunit\theta_*}{:}\ r\in [1, R]\}$,
  $C = \{e^{-\iunit\theta}{:}\ \theta\in [0, \theta_*]\}$, and $C_R =
  \{R e^{-\iunit \theta}{:}\ \theta\in  [0,\theta_*]\}$, similar
  argument leads to 
  \begin{align*}
    \int^\infty_1 e^{-\iunit r} r^q e^{-s r^\alpha \xi}\,\dd r
    \to
    -\iunit \int^{\theta_*}_0 e^{-\iunit e^{-\iunit\theta} -
    \iunit(q+1)\theta}\,\dd\theta+
    e^{-\iunit (q+1)\theta_*} \int^\infty_1 e^{-\iunit r
    e^{-\iunit\theta_*}} r^q \,\dd r
  \end{align*}
  uniformly in $\xi\in D$.

  In each of the above convergence results, the limit is independent
  of $\xi$.  Thus, as $s\to0+$, $f_s(\xi)$ converges to the same
  constant $\lambda$ uniformly in any compact set of
  $\xi\in\{z{:}~\Re z>0\}$.  It only remains to find out $\lambda$.
  For this purpose, fix $\xi=1$.  For $t>0$, let
  \begin{align*}
    f(s,t) = \intzi J_\tau(r) r^{p-1} e^{-s r^\alpha - t r}\,\dd r.
  \end{align*}
  Then exactly the same argument as above shows that as $(s,t)\to (0+,
  0+)$, $f(s,t)\to \lambda$.  On the other hand, given $t>0$, by
  dominated convergence, as $s\to0+$,
  \begin{align*}
    f(s,t)\to f(0,t) = 
    \intzi J_\tau(r) r^{p-1} e^{-t r}\,\dd r.
  \end{align*}
  Then $\lambda = \lim_{t\to0+} f(0,t)$.  From \cite{andrews:99:cup},
  Eq.~(9.10.5), the last limit is the \rhs of \eqref{e:Bessel-limit}. 
\end{proof}

\begin{proof}[Proof of Lemma \ref{l:Gtlimit}]
  By \eqref{e:proj-Gt} and dominated convergence, for $j,k\in\Ints_+$,
  \begin{align*}
    &
    \frac{\dd^k}{\dd t^k} (\proj_j G_t)(u)
      =
      \frac{(-\iunit)^j c_{j,d}}{A(\sph{d-1})} 
      \frac{\dd^k}{\dd t^k}    
      \sint \wt C^b_j(\ip u v) \Sbr{
      \intzi r^{d/2}J_{j+b}(r) e^{-t r^\alpha V(v)}\,\dd r
      }\omega(\dd v)
    \\
    &\hspace{1cm}=
      \frac{(-1)^k  (-\iunit)^j c_{j,d}}{A(\sph{d-1})}
      \sint \wt C^b_j(\ip u v) V(v)^k \Sbr{
      \intzi r^{d/2+k\alpha}J_{j+b}(r) e^{-t r^\alpha V(v)}\,\dd r
      }\omega(\dd v).
  \end{align*}
  By Lemma \ref{l:Bessel-limit}, as $t\to0+$,
  \begin{align*}
    \intzi r^{d/2+k\alpha}J_{j+b}(r) e^{-t r^\alpha V(v)}\,\dd r
    \to
    \frac{2^{d/2+k\alpha}
    \Gamma((j+k\alpha+d)/2)}{\Gamma((j-k\alpha)/2)}
  \end{align*}
  uniformly in $v\in\sph{d-1}$.  Then
  \begin{align*}
    &
      \frac{\dd^k}{\dd s^k}(\proj_{j,d}G_s)(u)
    \\
    &\to
      \frac{(-1)^k}{A(\sph{d-1})} (-\iunit)^j c_{j,d}
      \sint \wt C^b_j(\ip u v) V(v)^k
      \frac{2^{d/2+k\alpha}
      \Gamma((j+k\alpha+d)/2)}{\Gamma((j-k\alpha)/2)}
      \omega(\dd v)
    \\
    &=
      \frac{2^{d/2}(-2^\alpha)^k  \Gamma((j+k\alpha+d)/2)}
      {\Gamma((j-k\alpha)/2)} S_{j,k}(u)
  \end{align*}
  uniformly in $u\in\sph{d-1}$.
\end{proof}

\section{An example with linear spectral spherical
  density}  \label{s:example}
Fix $c\ge1$ and $\theta\in\sph{d-1}$, where $d\ge2$.  This section
applies the results in previous sections to the case where $\mu$ is an
$\alpha$-stable distribution on $\Reals^d$ with spectral spherical
density 
\begin{align*}
  P(u) = c + \ip \theta u
\end{align*}
and no shift.  This is the simplest case other than the one that
assumes a constant spectral spherical density.  Similar treatment also
applies if the spectral spherical density is $F(\ip\theta \cdot)$.
Denote the zonal harmonic $\wt C^{(d-2)/2}_j(\ip\theta\cdot)$ by  $Z_j(\cdot)$.

\begin{prop} \label{p:linear}
  Let
  \begin{align} \label{e:linear-kappa}
    \kappa_0 = \frac{c}{\Gamma(\frac{\alpha+d}2)
    \Gamma(\frac\alpha2+1)},
    \quad
    \kappa_1 = -\frac{\tan(\frac{\pi\alpha}{2})}
    {\Gamma(\frac{\alpha+d+1} 2) \Gamma(\frac{\alpha+1}2)},
    \quad
    \kappa_2 = \frac{\pi^{d/2}\Gamma(\alpha+1)}{2^{\alpha -
    1}}.
  \end{align}
  \begin{remarklist}
  \item\label{p:linear-che}
    If $\alpha\ne1$, then the characteristic exponent of $\mu$ is
    \begin{align} \label{e:che-linear}
      \Phi_\mu(z) = |z|^\alpha V(u_z) \quad\text{with}\quad
      V(u)
      =
      \kappa_3[\kappa_0 + \iunit \kappa_1 \ip\theta{u}],
    \end{align}
    and if $\alpha=1$, then
    \begin{align}\label{e:che-linear1}
      \Phi_\mu(z)
      =
        2|z| \pi^{(d-1)/2} \Sbr{
        \frac{c}{\Gamma(\frac{1+d}2)}
        -\frac{\iunit\beta_d \ip\theta{u_z}}
        {\Gamma(\frac d2+1)\sqrt\pi} }
        + \iunit|z|\ln|z| \frac{2\pi^{d/2-1}}{\Gamma(\frac d2+1)}
        \ip\theta{u_z},
    \end{align}
    where $\beta_d$ is given in Theorem \ref{t:che}.
  \item\label{p:linear-pdf}
    If $\alpha\in (0,1)$, then
    the \pdf of $\mu$ at $x\ne0$ is
    \begin{align} \label{e:pdf-linear<1}
      g(x)
      =
      \frac{\Gamma(\frac d2)}{\pi^{d/2}}
      \sumoi k \frac{[-2\Gamma(\alpha+1) \pi^{d/2}]^k}
      {k! |x|^{k\alpha+d}}
      \sum^k_{j=0} c_{j,d} \gamma_{j,k} Z_j(u_x),
    \end{align}
    where
    \begin{align*}
      \gamma_{j,k}
      =
      \frac{\Gamma(\frac{j+k\alpha+d}2)}{
      \Gamma(\frac{j-k\alpha}2)}
      \sum^{\Flr{(k-j)/2}}_{m=0}
      \frac{k!}{(k-j-2m)! \Gamma(j+m+\frac d2) m!} 
      \frac{ (-1)^m \kappa^{k-j-2m}_0 \kappa^{j+2m}_1} {2^{j+2m}},
    \end{align*}
    and if $\alpha\in (1,2)$, then the \pdf at any $x$ is
    \begin{align}\nonumber
      g(x)
      &=
        \frac{2^{1-d} \Gamma(\frac d2)}{\pi^{d/2}\alpha}
        \sumzi n
        \kappa^{-(n+d)/\alpha}_2
        \Grp{\frac{|x|}{2}}^n
        \sum^{\Flr{n/2}}_{m=0} \frac{(-1)^{m+n}
        c_{n-2m,d}}{m!\Gamma(n-m+\frac d2)}\times
      \\\label{e:pdf-linear>1}
      &\quad
        \sum^\infty_{k=n-2m}
        \frac{\Gamma(k+(n+d)/\alpha)}{
        (\kappa_0 + \kappa^2_1/\kappa_0)^{k+(n+d)/\alpha}}
        \gamma^*_{n-2m,k} Z_{n-2m}(u_x),
    \end{align}
    where for $k\ge j$,
    \begin{align*}
      \gamma^*_{j,k}
      =
      \sum^{\Flr{(k-j)/2}}_{l=0}
      \frac{1}{(k-j-2l)! \Gamma(j+l+\frac d2) l!}
      \frac{(-1)^l \kappa^k_1
      (\kappa_1/\kappa_0)^{k-j-2l}} 
      {2^{j+2l}}.
    \end{align*}
  \end{remarklist}  
\end{prop}

To prove Proposition \ref{p:linear}, the next technical result is
needed.
\begin{lemma} \label{l:proj-1form}
  Let $f_k(u) = \ip\theta u^k$.  Then for
  $j$, $k\in\Ints_+$ and $u\in \sph{d-1}$,
  \begin{align*}
    \proj_j f_k = 
    \begin{cases}
      \displaystyle
      \frac{c_{j,d} w_j(k)}{A(\sph{d-1})} Z_j,
      & \text{if $k-j\ge0$ is even} \\
      0 & \text{else,}
    \end{cases}
  \end{align*}
  where $w_j(k)$ are given in \eqref{e:Funk-Hecke2}.
\end{lemma}
\begin{proof}
  If $j>k$, then $\proj_j f_k=0$ as $f_k\in\Cal P_{k, d}$.  If
  $j\le k$, then by \eqref{e:proj} and \eqref{e:Funk-Hecke0}, for
  $u\in\sph{d-1}$,
  \begin{align*}
    (\proj_j f_k)(u)
    &=
      \frac{c_{j,d}}{A(\sph{d-1})}
      \sint
      \wt C^{(d-2)/2}_j(\ip u v) \ip \theta v^k\, \omega(\dd v)
    \\
    &=
      \frac{c_{j,d}}{A(\sph{d-1})}
      \lambda_{j,k} \wt C^{(d-2)/2}_j(\ip \theta u),
  \end{align*}
  where
  \begin{align*}
    \lambda_{j,k}
    =
    A(\sph{d-1}) \int^1_{-1} t^k \wt C^{(d-2)/2}_j(t)
    (1-t^2)^{(d-3)/2}\,\dd t.
  \end{align*}
  Since $\wt C^{(d-2)/2}_j(-t) = (-1)^j\wt C^{(d-2)/2}_j(t)$, if $j$
  and $k$ have different parities, then $\lambda_{j,k}=0$.  On the
  other hand, if $j$ and $k$ have the same parity, then by comparing
  the above display and \eqref{e:F-H-G}, $\lambda_{j,k}$ is exactly
  $w_j(k)$.
\end{proof}
\begin{proof}[Proof of Proposition \ref{p:linear}]
  Denote $f(u) = \ip\theta u$.

  \medbreak\noindent
  \ref{p:linear-che} From $P = c + f$ and $f\in \Cal H_{1,d}$, $P_0 =
  \proj_0 P = c$, $P_1 = \proj_1 P = f$, and $P_j = \proj_j P = 0$ for
  $j>1$.   Then \eqref{e:che-linear} and \eqref{e:che-linear1}
  directly follow from Theorem \ref{t:che}.

  \medbreak\noindent
  \ref{p:linear-pdf}
  Let $\alpha\in (0,1)\cup(1,2)$.  If $k<j$, then by $(\kappa_0
  + \iunit \kappa_1 f)^k\subset \Cal P_{k,d}$, $\proj_j[(\kappa_0 +
  \iunit \kappa_1 f)^k]=0$.  On the other hand, if $k\ge j$, then by
  Lemma \ref{l:proj-1form},
  \begin{align*}
    \proj_j[(\kappa_0 + \iunit \kappa_1 f)^k]
    &=
      \sum^k_{l=0} \binom k l \kappa^{k-l}_0 (\iunit \kappa_1)^l
      \proj_j f_l
    \\
    &=
      \sum_{j\le l\le k,\ l-j{~\rm even}}
      \binom k l \kappa^{k-l}_0 (\iunit \kappa_1)^l
      \frac{c_{j,d} w_j(l)}{A(\sph{d-1})} Z_j
    \\
    &=
      \frac{\iunit^j c_{j,d} Z_j}
      {A(\sph{d-1})}
      \sum^{\Flr{(k-j)/2}}_{m=0}
      \binom k {j+2m} \kappa^{k-j-2m}_0 (-1)^m \kappa^{j+2m}_1
      w_j(j+2m).
  \end{align*}
  Apply the first expression in \eqref{e:Funk-Hecke2} to $w_j(j+2m)$
  on the \rhs.  It follows that
  \begin{align}\nonumber
    &
      \proj_j[(\kappa_0 + \iunit \kappa_1 f)^k]
    \\\label{e:proj-poly}
    &=
      \iunit^j c_{j,d}\Gamma(d/2) Z_j
      \sum^{\Flr{(k-j)/2}}_{m=0}
      \frac{k!}{(k-j-2m)! \Gamma(j+m+\frac d2) m!} 
      \frac{ (-1)^m \kappa^{k-j-2m}_0 \kappa^{j+2m}_1}{2^{j+2m}}.
  \end{align}
  
  Let $\alpha\in (0,1)$.  Then by Theorem \ref{t:pdf-lt1} and
  \eqref{e:che-linear}, for $x\ne0$,
  \begin{align*}
    g(x)
    =
    \comment{
    \sumoi k \frac{(-2\alpha)^k \pi^{-d/2}}{k! |x|^{k\alpha+d}}
    \sum^k_{j=0}
    \frac{\Gamma(\frac{j+k\alpha+d}2)}{\Gamma(\frac{j-k\alpha}2)}
    \Grp{\frac{\pi^{d/2} \Gamma(\alpha+1)}{2^{\alpha-1}}}^k
    (-\iunit)^j \proj_j[(\kappa_0 + \iunit\kappa_1 \ip\theta
    \cdot)^k](u_x)
    }
    \sumoi k \frac{[-2\Gamma(\alpha+1) \pi^{d/2}]^k}{\pi^{d/2} k!
    |x|^{k\alpha+d}}
    \sum^k_{j=0}
    \frac{\Gamma(\frac{j+k\alpha+d}2)}{\Gamma(\frac{j-k\alpha}2)}
    (-\iunit)^j \proj_j[(\kappa_0 + \iunit\kappa_1 f)^k](u_x).
  \end{align*}
  Then by \eqref{e:proj-poly}, \eqref{e:pdf-linear<1} follows.
  
  Let $\alpha\in (1,2)$.  From Lemma \ref{e:V-R} and
  \eqref{e:che-linear}, $\sup_{\sph{d-1}} |R - V|<R$ for any
  \begin{align*}
    R>\varrho_0 = \sup_{z\in\sph{d-1}} \frac{|V|^2}{2\Re(V)}
    =
    \kappa_2
    \frac{\kappa^2_0 + \kappa^2_1}{2\kappa_0}.
  \end{align*}
  Choose $R = 2\varrho_0$, so that $V_* = R - V = 
    \kappa_2[\kappa^2_1/\kappa_0 - \iunit\kappa_1 f]$.
  Then by Theorem \ref{t:pdf-gt1},
  \begin{align*}
    g(x)
    &=\frac{2^{1-d}}{\pi^{d/2}\alpha}
      \sumzi n \kappa^{-(n+d)/\alpha}_2\Grp{\frac{|x|}{2}}^n
      \sum^{\Flr{n/2}}_{m=0} \frac{(-1)^m}
      {m!\Gamma(n-m+\frac d2)}\times
    \\
    &\quad
      \sum^\infty_{k=n-2m} \frac{\Gamma(k+ \frac{n+d}\alpha)}
      {k! (\kappa_0 + \kappa^2_1/\kappa_0)^{k+(n+d)/\alpha}}
      (-\iunit)^{n-2m} \proj_{n-2m}[(\kappa^2_1/\kappa_0 -
      \iunit\kappa_1 f)^k](u_x).
  \end{align*}
  Then from \eqref{e:proj-poly}, but with $\kappa_0$ and $\kappa_1$
  therein replaced with $\kappa^2_1/\kappa_0$ and $-\kappa_1$,
  respectively, \eqref{e:pdf-linear>1} follows.
\end{proof}

\section{Sampling issues} \label{s:sampling}
The results in Section \ref{s:pdf(0,1)} and \ref{s:pdf(1,2)}
allow sampling with the series method, which is a version of the
rejection method (\cite{devroye:86:sv-ny}, sections II.2 and IV.5).
In general, suppose we wish to sample from a \pdf $g$ that is 
specified as $\const\times g^*$, where $g^*\ge0$ is known.  The
multiplicative factor need not be specified and is often intractable.
The rejection method exploits a \pdf $f$ that is relatively easy to
sample.  Similar to $g$, $f$ is specified as $\const\times f^*$, where
the so-called dominating function $f^*$ satisfies $g^*\le f^*$.
Denote by $U$ a random variable uniformly distributed on $(0,1)$.
Then $f$ can be sampled as follows.

\begin{proof}[Rejection method]{\ }
  \begin{itemize}[leftmargin=3ex, parsep=0ex, itemsep=0ex, topsep=.4ex]
  \item Keep sampling $X\sim f$ and $U$ independently until $U f^*(X)
    \le g^*(X)$.  Then output $X$ and stop.  \qedhere
  \end{itemize}
\end{proof}

As a version of the rejection method, the series method deals with
$g^*(x)$ and $f^*(x)$ that are specified as convergent series
$\sumoi k a_k(x)$ and $\sumoi k b_k(x)$, respectively, where $a_k(x)$
and $b_k(x)$ are real valued.  Write $g^*_k(x) = \sum^k_{n=1} a_n(x)$
and $f^*_k(x) = \sum^k_{n=1} b_n(x)$.  Suppose functions $A_{k+1}(x)$
and $B_{k+1}(x)$ are available, such that 
\begin{align*}
  |g^*(x) - g^*_k(x)| \le A_{k+1}(x), \quad
  |f^*(x) - f^*_k(x)| \le B_{k+1}(x)
\end{align*}
and $A_{k+1}(x)\to 0$ and $B_{k+1}(x)\to0$ as $k\toi$.  Then $f$ can
be sampled as follows.

\begin{proof}[Series method]{\ }
  \begin{enumerate}[itemsep=0ex, parsep=0ex, leftmargin=3ex,
    topsep=.4ex]
  \item Sample $X\sim f$ and $U$ independently.
  \item Find the first $k$ such that $|g^*_k(X) - Uf^*_k(X)| >
    A_{k+1}(X) + U B_{k+1}(X)$.
  \item If $g^*_k(X)>U f^*_k(X)$, then output $X$ and stop, otherwise
    go back to step 1 and repeat.  \qedhere
  \end{enumerate}
\end{proof}
Note that if $a_1(x) = g^*(x)$, $b_1(x) = f^*(x)$, and $a_k(x) =
b_k(x) = A_{k+1}(x) = B_{k+1}(x) \equiv0$ for $k>1$, then the series
method reduces to the rejection method.  To verify that its output
follows $f$, with probability one, $|g^*(X) - U f^*(X)|>0$, so for
large $k$, $|g^*_k(X) - U f^*_k(X)| > A_{k+1}(X) + U B_{k+1}(X)$.  If
$g^*_k(X) > U f^*_k(X)$, then we have
$g^*_k(X) - U f^*_k(X) > A_{k+1}(X) + U B_{k+1}(X)$, and so
\begin{align*}
  g^*(X) \ge g^*_k(X) - A_{k+1}(X)
  > U f^*_k(X) + U B_{k+1}(X) \ge U f^*(X).
\end{align*}
Likewise, if $g^*_k(X) < U f^*_k(X)$, then $g^*(X) < U f^*(X)$.
Thus, the above algorithm implements the rejection method.

\subsection{Case 1: $\alpha\in (0,1)$}
The series method will rely on the following.
\begin{prop} \label{p:bound<1}
  Let the conditions in Theorem \ref{t:pdf-lt1} be satisfied.  That
  is, $\mu$ is an $\alpha$-stable distribution on $\Reals^d$, $d\ge2$,
  with $\alpha\in(0,1)$ and characteristic exponent \eqref{e:che-ss3},
  and $\lambda = P\omega$ with $0\ne P\in \Cal P_{\dg,d}$ for some
  $\dg\in\Ints_+$.  Then
  \begin{align} \label{e:bound<1}
    g(x) \le
    C_1 \cf{|x|\le1} + \frac{C_2}{|x|^{\alpha+d}} \cf{|x|>1},
  \end{align}
  where $C_1 = \frac{\Gamma(d/\alpha)}{\alpha (2\pi)^d}
  \Sbr{\inf_{\sph{d-1}}\Re(V)}^{-d/\alpha}$ and 
  \begin{align*}
    C_2 =
    \nth{\pi^{d/2}} \sumoi k \frac{(2^\alpha)^k}{k!}
    \sup_{\sph{d-1}}|V|^k \sum^{k\dg}_{j=0}
    \frac{\sqrt{c_{j,d}}\Gamma((j+k\alpha+d)/2)}
    {|\Gamma((j-k\alpha)/2)|}.
  \end{align*}
\end{prop}
\begin{proof}
  From \eqref{e:ift-ss}, \eqref{e:che-ss3}, and Lemma \ref{l:ReV}, for
  any $x\in\Reals^d$,
  \begin{align*}
    g(x)
    \le
    (2\pi)^{-d} \rint e^{-|z|^\alpha \Re(V(z))}\,\dd z
    \le
    (2\pi)^{-d} \intzi r^{d-1} e^{-r^\alpha \inf_{\sph{d-1}}
    \Re(V)}\,\dd r = C_1.
  \end{align*}
  On the other hand, from Theorem \ref{t:pdf-lt1}, for $x\ne0$,
  \begin{align*}
    g(x)
    \le
    \sumoi k \frac{(2^\alpha)^k \pi^{-d/2}}{k! |x|^{k\alpha+d}}
    \sum^{k\dg}_{j=0}
    \frac{\Gamma((j+k\alpha+d)/2)}{|\Gamma((j-k\alpha)/2)|}
    \sup_{\sph{d-1}}|S_{j,k}|.
  \end{align*}
  When $|x|>1$, by \eqref{e:Sjk-bound}, the \rhs is bounded by
  $C_2/|x|^{\alpha + d}$.  Then \eqref{e:bound<1} follows.
\end{proof}

From Proposition \ref{p:bound<1}, $g$ can be sampled by the series
method with the following inputs.
\begin{enumerate}[itemsep=0ex, parsep=0ex, leftmargin=3ex, topsep=1ex]
\item $g^*(x) = g(x) = \sumoi k a_k(x)$, where
  \begin{align*}
    a_k(x) =
    \frac{(-2^\alpha)^k \pi^{-d/2}}{k! |x|^{k\alpha+d}}
    \sum^{k\dg}_{j=0}
    \frac{\Gamma((j+k\alpha+d)/2)}{\Gamma((j-k\alpha)/2)}
    S_{j,k}(u_x),
  \end{align*}
  and the bound $A_{k+1}(x)$ is any convenient upper bound of
  \begin{align*}
    \nth{\pi^{d/2}} \sum^\infty_{n=k+1} \frac{(2^\alpha)^n}{n!
    |x|^{n\alpha+d}} \sup_{\sph{d-1}}|V|^k \sum^{n\dg}_{j=0}
    \frac{\sqrt{c_{j,d}}\Gamma((j+n\alpha+d)/2)}
    {|\Gamma((j-k\alpha)/2)|}
  \end{align*}
  as long as $A_{k+1}(x)\to0$ as $k\toi$.
\item $f^*(x) = D_1 \cf{|x|\le1} + D_2|x|^{-\alpha-d} \cf{|x|>1}$,
  where $D_1\ge C_1$ and $D_2\ge C_2$ are bounds that are easy to
  evaluate, and the series $\sumoi k b_k(x)$ for $f^*$ and
  corresponding bounds on remainders are given by $b_1(x) =
  f^*(x)$ and $b_k(x)=B_{k+1}(x)=0$ for $k>1$.
\end{enumerate}

With these inputs, in each iteration, the series method needs to
sample from the \pdf $f = f^*/\int f^*$, which is not hard because $f$
is a mixture of the uniform \pdf in the unit ball and the \pdf of
$U^{-1/\alpha} \vartheta$, where $U$ and $\vartheta$ are independent
random variables, with $U$ uniformly distributed on $(0,1)$ and
$\vartheta$ uniformly distributed on $\sph{d-1}$.  Furthermore, for
$X\sim f$, with probability 1, $X\ne0$, so all $a_k(X)$ and
$A_{k+1}(X)$ are well defined.

\subsection{Case 2: $\alpha\in(1,2)$}
The sampling when $\alpha\in(1,2)$ is more complicated because
the series in Proposition \ref{p:pdf-gt1} and Theorem \ref{t:pdf-gt1}
do not indicate a simple integrable function to dominate the \pdf.
Instead, for $X\sim\mu$, we will first sample $|X|$,  and then sample
$u_X$ conditioning on $|X|$.

To sample $|X|$, the case where $\mu$ is symmetric turns out to be
quite simple.  From Lemma \ref{l:ReV} and \eqref{e:V-sphere}, $V$ is
real valued and positive in this case.  Let $\eno Z d$, and $\zeta$
be independent random variables with $Z_i\sim N(0,1)$ and $\zeta$
positive and $(\alpha/2)$-stable such that $\mean(e^{-t\zeta}) =
e^{-t^{\alpha/2}}$ for $t\ge0$.  Put $Z=(\eno Z d)$.

\begin{prop} \label{p:sampling-sym>1}
  Let $X$ have the \pdf $g$ in Proposition \ref{p:pdf-gt1}.  If $g$ is
  symmetric, i.e., $g(x) =g(-x)$, then $V^{1/\alpha}(Z/|Z|)
  \sqrt{2\zeta} |Z|\sim |X|$.
\end{prop}
\begin{proof}
  From Corollary \ref{c:pdf-gt1} as well as the discussion preceding
  it, 
  \begin{align*}
    \frac{A(\sph{d-1}) (r/2)^{d-1}}{2\pi^{d/2}}
    \mean\Sbr{\zeta^{-d/2}
    \exp\Cbr{-\frac{r^2}{4\zeta}}}
  \end{align*}
  is the \pdf of $\sqrt{2\zeta} |Z|$.  Then from
  \eqref{e:pdf-abs-gt1b},
  \begin{align*}
    \pr\{|X|\in \dd r\}
    =
    \nth{A(\sph{d-1})} \sint \pr\{V^{1/\alpha}(u)
    \textstyle\sqrt{2\zeta}|Z|\in\dd r\}\,\omega(\dd u).
  \end{align*}
  The \rhs can be written as $\pr\{V^{1/\alpha}(\vartheta)
  \sqrt{2\zeta} |Z|\in \dd r\}$, where $\vartheta$ is uniformly
  distributed on $\sph{d-1}$ and independent of $(|Z|, \zeta)$.
  Since $Z/|Z|$ is uniformly distributed on $\sph{d-1}$ and
  independent of $(|Z|,\zeta)$, then $\vartheta$ can be replaced with
  $Z/|Z|$ and the claim follows.
\end{proof}

For the general case, the series method is based on the following.
\begin{prop} \label{p:sampling-norm>1}
  Let the conditions in Theorem \ref{t:pdf-gt1} be satisfied.  That
  is, $\mu$ is a nondegenerate stable distribution on $\Reals^d$,
  $d\ge2$, with $\alpha\in(1,2)$ and characteristic exponent
  \eqref{e:che-ss3}.  Let $w = \inf_{\sph{d-1}} \Re(V^{-2/\alpha})$.
  \begin{enumerate}[itemsep=0ex, parsep=0ex, leftmargin=3ex,
    topsep=.4ex, label=\arabic*)]
  \item \label{i:norm>1a}
    If $w>0$, then letting $V_0 = w^{-\alpha/2}$ and $f$ be the \pdf
    of $V^{1/\alpha}_0 \sqrt{2\zeta}|Z|$, for $X\sim g$,
    \begin{align*}
      g_{|X|}(r) \le C f(r) \quad\text{with}\quad C = \sup_{\sph{d-1}}
      |V_0/V|^{d/\alpha}. 
    \end{align*}
  \item \label{i:norm>1b} Suppose $\alpha\in [4/3,2)$ and $\mu$ has a
    polynomial spectral spherical density.  Then $w>0$.
  \end{enumerate}
\end{prop}

\begin{proof}
  \ref{i:norm>1a}  From \eqref{e:pdf-abs-gt1b},
  \begin{align*}
    g_{|X|}(r)
    &\le
      \frac{(r/2)^{d-1}}{2\pi^{d/2}}
      \sint \nth{|V|^{d/\alpha}}\mean\Sbr{\zeta^{-d/2}
      e^{-\Re(V^{-2/\alpha}) r^2/4\zeta}}\,\dd\omega
    \\
    &\le
      \sup_{\sph{d-1}}|V_0/V|^{d/\alpha}
      \times \frac{(r/2)^{d-1}}{2\pi^{d/2}}
      \sint \nth{V^{d/\alpha}_0}\mean\Sbr{\zeta^{-d/2}
      e^{- V^{-2/\alpha}_0 r^2/4\zeta} }\,\dd\omega.
  \end{align*}
  Then from Proposition \ref{p:sampling-sym>1} the proof follows.
  
  \ref{i:norm>1b} From Corollary \ref{c:proj-ineq}, the assumption
  implies $\sup_{\sph{d-1}}|\arg(V^{2/\alpha})| <
  \pi(1-\alpha/2)(2/\alpha) \le\pi/2$ and so the proof follows by
  noticing $\Re(V^{-2/\alpha}) = |V|^{-2/\alpha}
  \cos(\arg(V^{2/\alpha}))$.
\end{proof}

From Proposition \ref{p:sampling-norm>1}, provided $w=\inf_{\sph{d-1}}
\Re(V^{-2/\alpha})>0$, $g_{|X|}$ can be sampled using the series
method with $f^*=C f$ as the dominating function.  However, while $f$
is easy to sample, neither $g_{|X|}$ nor $f^*$ has a closed form.  We
will rely on the series representations in Proposition
\ref{p:pdf-gt1} and Theorem \ref{t:pdf-gt1}.   The inputs to the
series method can be as follows.
\begin{enumerate}[itemsep=0ex, parsep=0ex, leftmargin=3ex, topsep=1ex]
\item $g^*(r)=g_{|X|}(r) = \sumzi k a_k(r)$, where $a_k(r)$ is any
  ordering of the terms in the \ac double series in
  \eqref{e:pdf-abs>1}.  As noted after Theorem \ref{t:pdf-gt1}, when
  $\mu$ has a polynomial spectral spherical density, all the terms can
  be expressed in closed form.
\item $f^*(r) = C f(r) = \sumzi k b_k(r)$, where $C$ and $f$ are as in
  Proposition \ref{p:sampling-norm>1}, and by \eqref{e:pdf-abs-gt1},
  \begin{align*}
    b_k(r) = 
    \frac{C A(\sph{d-1})}{\pi^{d/2}\alpha}
    \frac{(-1)^k \Gamma((2k+d)/\alpha)}{k! \Gamma(k+d/2)}
    \Grp{\frac{r}{2}}^{2k+d-1} \nth{V^{(2k+d)/\alpha}_0},
  \end{align*}
  where $V_0 = w^{-\alpha/2}$.
\end{enumerate}

Suppose $|X|$ has been sampled, the next step is to sample $u_X$
conditional on $|X|$.  Suppose $\mu$ has a polynomial spectral
spherical density $P\in\Cal P_{\dg,d}$.  Then from Theorem
\ref{t:pdf-gt1}, the \pdf of $u_X$ \wrt $\omega(\dd u)$ given $|X|=r$
is in proportion to 
\begin{align*}
  g^*(u)
  &= \sumzi {k,m} \sum^{2m + k\dg}_{n=2m} 
    \frac{(-1)^m (r/2)^n}{m!\Gamma(n-m+d/2)}
    \frac{\Gamma(k+(n+d)/\alpha)}{k! R^{k+(n+d)/\alpha}}
    S^*_{n-2m,k}(u)
  \\
  &\le
    C=\sumzi {k,m} \sum^{2m + k\dg}_{n=2m} 
    \frac{(r/2)^n}{m!\Gamma(n-m+d/2)}
    \frac{\Gamma(k+(n+d)/\alpha)}{k! R^{k+(n+d)/\alpha}}
    \sup_{\sph{d-1}} |S^*_{n-2m,k}|
    <\infty.
\end{align*}
Then the conditional distribution can be sampled using a constant $f^*$
as the dominating function, with the value of the constant being any
convenient upper bound of $C$.  The corresponding \pdf is uniform on
$\sph{d-1}$.  The way to describe the inputs to the series method
pretty much follows the one for $|X|$, so for brevity the description
is omitted.

\bibliographystyle{acmtrans-ims}

\end{document}